\documentclass[12pt,a4paper,leqno]{article}

\usepackage[centertags]{amsmath}
\usepackage{amsthm}
\usepackage{amssymb,exscale}
\usepackage{paralist}
\usepackage{latexsym}
\usepackage{graphicx}
\usepackage[perpage,multiple]{footmisc}
\usepackage{hyperref}

\swapnumbers
\theoremstyle{definition}
\newtheorem{theorem}[equation]{Theorem}
\newtheorem{lemma}[equation]{Lemma}
\newtheorem{corollary}[equation]{Corollary}
\newtheorem{definition}[equation]{Definition}
\newtheorem{proposition}[equation]{Proposition}
\newtheorem{remark}[equation]{Remark}

\newtheorem{innercustomprop}{Proposition}
\newenvironment{customprop}[1]
  {\renewcommand\theinnercustomprop{#1}\innercustomprop}
  {\endinnercustomprop}

\newcommand\fnsep{\textsuperscript{,}}

\renewcommand{\phi}{\varphi}

\newcommand{\D}{\mathrm{d}}
\newcommand{\E}{\mathrm{e}}

\newcommand{\ti}{\tilde}

\renewcommand{\(}{\bigl(}
\renewcommand{\)}{\bigr)\vphantom{)}}

\newcommand{\equi}{\>\Longleftrightarrow\>}

\newcommand{\impl}{\>\Longrightarrow\>}

\newcommand{\vol}{\operatorname{vol}}

\newcommand{\Spl}{\operatorname{Spl}}
\newcommand{\BB}{\operatorname{BB}}
\newcommand{\G}{\operatorname{G}}

\newcommand{\One}{{1\hskip-2.5pt{\rm l}}}

\newcommand{\ccirc}{\mathbin{\kern-0.13em\circ}}

\newcommand{\eps}{\varepsilon}

\newcommand{\si}{\sigma}
\newcommand{\ga}{\gamma}

\newcommand{\Om}{\Omega}
\newcommand{\de}{\delta}
\newcommand{\De}{\Delta}
\newcommand{\al}{\alpha}
\newcommand{\be}{\beta}
\newcommand{\cO}{\mathcal O}

\newcommand{\A}{\mathcal A}
\newcommand{\B}{\mathcal B}
\newcommand{\la}{\lambda}

\newcommand{\Ex}{\mathbb E\,}
\newcommand{\R}{\mathbb R}

\newcommand{\PR}[1]{\mathbb{P}\mskip1.5mu\bigg(\mskip1.5mu#1\mskip1.5mu\bigg)}

\newcommand{\sif}{$\sigma$\nobreakdash-\hspace{0pt}field}

\newcommand{\splittable}[1]{$#1$\nobreakdash-\hspace{0pt}splittable}
\newcommand{\splittability}[1]{$#1$\nobreakdash-\hspace{0pt}splittability}

\newcommand{\myatop}[2]{{\genfrac{}{}{0pt}{}{#1}{#2}}}

\begin{document}

\title{Linear response and moderate deviations:\\ hierarchical
  approach. IV}

\author{Boris Tsirelson}

\date{}
\maketitle

\begin{abstract}
The Moderate Deviations Principle (MDP) is well-understood for sums of
independent random variables, worse understood for stationary random
sequences, and scantily understood for random fields.
Here it is established for splittable random fields integrated against
test functions.
\end{abstract}

\setcounter{tocdepth}{2}
\tableofcontents

\numberwithin{equation}{section}

\section[Definition, and main result formulated]
  {\raggedright Definition, and main result formulated}
\label{sect1}
The definition of a splittable random field, used in \cite{II}
and \cite{III}, is geared toward integrals $ \int_B X_t \, \D t $ of
the random field $ (X_t)_{t\in\R^d} $ over boxes $ B \subset \R^d $
rather than more general integrals $ \int \phi(t)  X_t \, \D t $ with
a test function $ \phi : \R^d \to \R $. In order to examine integrals
with test functions, here we introduce a new definition of a
splittable random field.

We still do not need random variables $ X_t $; instead, by a random
field on $ \R^d $ we mean a random locally finite signed Borel measure
(LFSBM) on $ \R^d $, denoted (if only for convenience) by $ B \mapsto
\int_B X_t \, \D t $ where $B$ runs over bounded Borel measurable
subsets of $\R^d$. The set of all such measures (LFSBMs) is a standard
measurable space; its \sif\ is generated by the functions
$ \mu \mapsto \int f \, \D\mu $ where $ f $ runs over all compactly
supported continuous functions $ \R^d \to \R $ (or, equivalently, all
bounded Borel functions with bounded support; or only indicator
functions of boxes). Accordingly, a random LFSBM is a measurable map
from a given probability space to the space of LFSBMs; for convenience
we say ``random field'' instead of ``random LFSBM''.\footnote{%
 Thus, the centered Poisson point process is still an example, but the
 white noise is not. (Both are mentioned in \cite[page 2]{I} for $ d=1
 $.)}
Sometimes we specify $ d $ or/and $ \Om $, saying ``random field on
$ \R^d $'', or ``random field on  $ \R^d $ and $ \Om $'', or ``random
field on $ \Om $''. The notions ``independent'' and ``identically
distributed'' for such random fields  are interpreted naturally. Note
that independence makes sense only for random fields on the same $ \Om
$. In contrast, identically distributed random fields may live on
different $ \Om $ (but the same $ \R^d $).

\begin{definition}\label{1.1}
A \emph{split} of a random field $X$ (on $ \R^d $ and $ \Om $) is a
triple of random fields $ X^0, X^-, X^+ $ (all the three on $ \R^d $
and the same $ \ti\Om $ possibly different from $ \Om $) such that
the two random fields $ X^-, X^+ $ are (mutually) independent and the
four random fields $ X, X^0, X^-, X^+ $ are identically distributed.
\end{definition}

Informally, a split is useful when its leak, defined below, is
small.\footnote{%
 The same is written in \cite[p.~2]{II}; see Sect.~\ref{sect5} below
 (in particular, Prop.~\ref{5.2} and Lemma \ref{5.3}) for relations
 between the notions ``split'' and ``leak'' here and in \cite{II}.}

\begin{definition}\label{1.2}
(a) Let $X$ be a random field on $\R^d$ and $\Om$, and $ (X^0, X^-,
X^+) $ a split of $X$ on $\ti\Om$. The \emph{leak} of this split is
the random field $Y$ on $\R^d$ and $\ti\Om$ such that\footnote{%
 It is convenient to formulate relations in terms of (generally
 nonexistent) random variables $X_t$ when it is immediate to rewrite
 these relations in terms of LFSBMs.}\fnsep\footnote{%
 Typically, $ Y_{t_1,\dots,t_d} $ decays for large
 $|t_1|$.}\fnsep\footnote{%
 In \cite[p.~2]{II} the leak is taken along a coordinate hyperplane
 $ \{(t_1,\dots,t_d): t_k=r\} $; in contrast, here we restrict
 ourselves to $k=1$ and $r=0$, which simplifies Def.~\ref{1.2} but
 complicates Def.~\ref{1.3}(c).}
\[
Y_{t_1,\dots,t_d} = \begin{cases}
  X^-_{t_1,\dots,t_d} - X^0_{t_1,\dots,t_d} &\text{when } t_1<0,\\
  X^+_{t_1,\dots,t_d} - X^0_{t_1,\dots,t_d} &\text{when } t_1\ge0;
\end{cases}
\]
that is, $ \int_A Y_t \, \D t = \int_A (X^-_t-X^0_t) \, \D t $ for
every bounded Borel measurable set $ A \subset
(-\infty,0)\times\R^{d-1} $, and $ \int_A Y_t \, \D t = \int_A
(X^+_t-X^0_t) \, \D t $ for every bounded Borel measurable set $
A \subset [0,\infty)\times\R^{d-1} $.

(b) Let $X$ be a random field on $\R^d$ and $\Om$. A random field $Y$
on $ \R^d $ and $\Om_1$ is a \emph{leak} for $X$, if there exists a
split of $X$ on some $\ti\Om$ such that the leak of this split is
distributed like $Y$.\footnote{%
 That is, $Y$ and this leak are identically distributed.}
\end{definition}

Ultimately we are interested in stationary (that is, shift-invariant
in distribution) random fields, but for now we waive stationary,
getting in exchange a useful property, see Item 1 in the numbered list
below.

The variation of a LFSBM is a locally finite Borel measure (positive,
not signed). Thus, the variance of a random field is a random locally
finite Borel measure; we denote it $ B \mapsto \int_B |X_t| \, \D t $.

We'll define the notion ``splittable random field'' satisfying the
following.
\begin{compactenum}
\item If $ (X_t)_{t\in\R^d} $ is splittable and $ \phi : \R^d \to \R $
is bounded Borel measurable, then $ (\phi(t)X_t)_{t\in\R^d} $ is
splittable.\footnote{%
 See Corollary~\ref{2.3}.}
\item If $ (X_{t_1,\dots,t_d})_{t_1,\dots,t_d\in\R} $ is splittable and
$ (k_1,\dots,k_d) $ is a permutation of $ (1,\dots,d) $, then $
(X_{t_{k_1},\dots,t_{k_d}})_{t_1,\dots,t_d\in\R} $ is
splittable (``permutation invariance'').\footnote{%
 See Corollary~\ref{2.6}.}
\item If $ (X_t)_{t\in\R^d} $ is splittable and $ s \in \R^d $, then $
(X_{t+s})_{t\in\R^d} $ is splittable (``shift
 invariance'').\footnote{%
 See Corollary~\ref{2.9}.}
\item If $ (X_t)_{t\in\R^d} $ is splittable and $ A \subset \R^d $
bounded Borel measurable, then $ \Ex \exp \eps \int_A |X_t| \, \D t
< \infty $ for some $ \eps>0 $,\footnote{%
 Follows from Def.~\ref{1.3}(a), since $ \Ex \exp \eps \int_{A\cup B}
 |X_t| \, \D t \le
 \Ex \( \exp \eps \int_A |X_t| \, \D t \)
 \( \exp \eps \int_B |X_t| \, \D t \) \le
 \( \Ex \exp 2\eps \int_A |X_t| \, \D t \)^{1/2}
 \( \Ex \exp 2\eps \int_B |X_t| \, \D t \)^{1/2} $.}
and $ \Ex \int_A X_t \, \D t = 0 $.\footnote{%
 See Def.~\ref{1.3}(b).}
\begin{sloppypar}
\item If $ (X_t)_{t\in\R^d} $ is splittable and compactly supported
(that is, $ \int_{\R^d\setminus B} |X_t| \, \D t = 0 $ a.s.\ for some
bounded Borel measurable $B$), then there exists a splittable $
(Y_s)_{s\in\R^{d-1}} $ such that $ \( \int_{-\infty}^\infty
X_{r,s} \, \D r \)_{s\in\R^{d-1}} $ is distributed like the sum of $
(Y_s)_{s\in\R^{d-1}} $ and two mutually \emph{independent} random
fields, one distributed like $ \( \int_{-\infty}^0 X_{r,s} \, \D
r \)_{s\in\R^{d-1}} $, the other distributed like $ \( \int_0^\infty
X_{r,s} \, \D r \)_{s\in\R^{d-1}} $.\footnote{%
 See Prop.~\ref{2.13}.}
\end{sloppypar}
\item In the case $ d=1 $ the previous item means existence of a
random variable $Y$ such that $ \int_{-\infty}^\infty X_r \, \D r $ is
distributed like the sum of $ Y $ and two mutually \emph{independent}
random variables, one distributed like $ \int_{-\infty}^0 X_r \, \D r
$, the other distributed like $ \int_0^\infty X_r \, \D r $, and
$ \Ex \exp \eps |Y| < \infty $ for some $ \eps>0 $.
\end{compactenum}

\begin{definition}\label{1.3}
Let $ d \in \{1,2,\dots\} $ and $ k \in \{1,\dots,d\} $. A random
field $ (X_t)_{t\in\R^d} $ is \splittable{(1,k)}, if

(a) for every $ s \in \R^k $,\footnote{%
 Of course, $ [0,1)^k+s $ stands for the unit cube shifted by
 $s$. Typically, $ X_{t_1,\dots,t_d} $ decays for large $
 t^2_{k+1}+\dots+t^2_d $.}
\[
\Ex \exp \int_{([0,1)^k+s)\times\R^{d-k}} |X_t| \, \D t \le 2 \, ;
\]

(b) $ \Ex X_t = 0 $ for all $ t \in \R^d $; that is, for every
bounded Borel measurable set $ A \subset \R^d $,
\[
\Ex \int_A X_t \, \D t = 0
\]
(this expectation is well-defined, since (a) implies $ \Ex \int_A
|X_t| \, \D t < \infty $);

(c) for every $ i \in \{1,\dots,k\} $ and every $ r \in \R $ the
random field $ \ti X $ defined by\
\[
\ti X_{t_1,\dots,t_d} = X_{t_2,\dots,t_i,t_1+r,t_{i+1},\dots,t_d}
\]
has a leak $ \ti Y $ such that the random field $ Y $ defined
by\footnote{%
 A leak $\ti Y$ for $\ti X$ (along $\{(t_1,\dots,t_d):t_1=0\}$) could
 be thought of as a leak for $X$ along $\{(t_1,\dots,t_d):t_i=r\}$
 (recall the last footnote to Def.~\ref{1.2}) taken at $
 (t_2,\dots,t_i,t_1+r,t_{i+1},\dots,t_d) $. Note that $ \ti
 Y_{(t_1,\dots,t_d)} $ switches at $ t_1=0 $; accordingly, $ Y $
 switches at $ t_k=0 $. Typically, $ Y_{(t_1,\dots,t_d)} $ decays for
 large $ t_k^2+\dots+t_d^2 $.}
\[
Y_{t_1,\dots,t_d} = \ti Y_{t_k,t_1,\dots,t_{k-1},t_{k+1},\dots,t_d}
\]
is \splittable{(1,k-1)}, provided that $ k>1 $; otherwise, if $ k=1 $,
$ Y $ is required to satisfy
\[
\Ex \exp \int_{\R^d} |Y_t| \, \D t \le 2 \, .
\]
\end{definition}

\begin{definition}\label{1.4}
(a) Let $ C \in (0,\infty) $. A random field $ X = (X_t)_{t\in\R^d} $
is \splittable{(C,k)}, if the random field $ \frac1C X = \(\frac1C
X_t\)_{t\in\R^d} $  is \splittable{(1,k)}.

(b) In the case $ k=d $ we say that $X$ is \splittable{C}.

(c) If $X$ is \splittable{C} for some $ C \in (0,\infty) $, then we
say that $X$ is splittable.
\end{definition}

\begin{theorem}[\emph{``linear response''}]\label{theorem1}
For every splittable stationary random field $X$ on $ \R^d $ there
exists a number (evidently unique) $ \si_X \in [0,\infty) $
such that for every continuous compactly supported function $ \phi
: \R^d \to \R $ holds
\[
\lim_\myatop{ r\to\infty, \la\to0 }{ \la\log^d r \to 0 }
\frac1{ r^d \la^2}
\log \Ex \exp \la \int_{\R^d} \phi \Big( \frac1r t \Big) X_t \, \D t
= \frac12 \|\phi\|^2 \si_X^2
\]
where $ \|\phi\|^2 = \int_{\R^d} \phi^2(t) \, \D t $.
\end{theorem}

\begin{corollary}[\emph{moderate deviations}]\label{1.6}
For $ X $ and $ \phi $ as above, if $ \|\phi\| \si_X \ne 0 $, then
\[
\lim_{\textstyle\myatop
 { r\to\infty, c\to\infty }
 { ( c\log^d r )^2 / r^d \to 0 }
}
\frac1{c^2} \log \PR{ \int_{\R^d} \phi \Big( \frac1r t \Big) X_t \, \D
 t \ge c \|\phi\| \si_X r^{d/2} } = -\frac12 \, .
\]
\end{corollary}

\begin{corollary}\label{1.7}
For $ X $ and $ \phi $ as above, if $ \|\phi\| \si_X \ne 0 $, then the
distribution of $ r^{-d/2} \int_{\R^d} \phi \( \frac1r t \) X_t \, \D
t $ converges (as $ r \to \infty $) to the normal distribution $
N ( 0, \|\phi\|^2 \si_X^2 ) $.
\end{corollary}

\section[Basic observations]
  {\raggedright Basic observations}
\label{sect2}
Notions of \splittability{(1,k)} and \splittability{(C,k)} are
used in this section only (for induction in $k$). Further, in Sections
\ref{sect3}--\ref{sect5}, we use only splittability and
\splittability{C}, and need only corollaries (rather than lemmas and
propositions) from this section.

\begin{proposition}\label{2.1}
If $ (X_t)_{t\in\R^d} $ is \splittable{(1,k)} and $ \phi : \R^d \to
[-1,1] $ is Borel measurable, then $ (\phi(t)X_t)_{t\in\R^d} $ is
\splittable{(1,k)}.
\end{proposition}

\begin{lemma}\label{2.2}
Let $ \phi : \R^d \to \R $ be a locally bounded  Borel measurable
function.

(a) If $(X^0,X^-,X^+)$ is a split of a random field $X$, then $(\phi
X^0,\phi X^-,\phi X^+)$ is a split of the random field $ \phi X =
(\phi(t)X_t)_{t\in\R^d} $.

(b) If $Y$ is the leak of the split $(X^0,X^-,X^+)$, then $\phi Y$ is
the leak of the split $(\phi X^0,\phi X^-,\phi X^+)$.

(c) If $Y$ is a leak for $X$, then $\phi Y$ is a leak for $\phi X$.
\end{lemma}

\begin{proof}
(a): just from Def.~\ref{1.1}; (b): just from (a) and
Def.~\ref{1.2}(a); (c): just from (b) and Def.~\ref{1.2}(b).
\end{proof}

\begin{proof}[Proof of Prop.~\ref{2.1}]
Induction in $k=1,\dots,d$. We inspect Def.~\ref{1.3}(a,b,c).

(a) Just $ |\phi(t)X_t| \le |X_t| $.

(b) First, for an indicator function $ \phi(t) = \One_B (t) $ just $
\int_A \phi(t) X_t \, \D t = \int_{A\cap B} X_t \, \D t $. Second,
take linear combinations of such $\phi$. Third, take limits of such
combinations.

(c) In the notation of Def.~\ref{1.3}(c) we have $ \widetilde{\phi X} =
\ti\phi \ti X $,  $ \ti \phi : \R^d \to [-1,1] $; by Lemma
\ref{2.2}(c), $ \ti\phi \ti X $ has a leak $ \ti\phi \ti Y $; instead
of $Y$ we get $\psi Y$, $ \psi : \R^d \to [-1,1] $; and $ \psi Y $ is
\splittable{(1,k-1)} by the induction hypothesis, provided that $ k>1
$; otherwise, for $k=1$, just $ |\psi(t)Y_t| \le |Y_t| $.
\end{proof}

\begin{corollary}\label{2.3}
If $ (X_t)_{t\in\R^d} $ is \splittable{C_1} and $ \phi : \R^d \to
[-C_2,C_2] $ Borel measurable, then $ (\phi(t)X_t)_{t\in\R^d} $ is
\splittable{C_1 C_2}.
\end{corollary}

\begin{proof}
$ \frac1{C_1} X $ is \splittable{(1,d)}; by Prop.~\ref{2.1}, $
\frac1{C_2} \phi \frac1{C_1} X $ is \splittable{(1,d)}.
\end{proof}

\begin{proposition}\label{2.4}
If $ (X_{t_1,\dots,t_d})_{t_1,\dots,t_d\in\R} $ is \splittable{(1,k)},
$ (\ell_1,\dots,\ell_k) $ is a permutation of $ (1,\dots,k) $, and $
s \in \R^d $, then
\[
(X_{t_{\ell_1}+s_1,\dots,t_{\ell_k}+s_k,
 t_{k+1}+s_{k+1},\dots,t_d+s_d})_{t_1,\dots,t_d\in\R}
\]
is \splittable{(1,k)}.
\end{proposition}

Given a random field $X$ on $\R^d$ and a homeomorphism $ \al
: \R^d \to \R^d $ that preserves Lebesgue measure, we introduce the
random field $ X \circ \al $ by $ \forall t\in\R^d \;\> (X\circ\al)_t
= X_{\al(t)} $; that is, $ \int_B (X\circ\al)_t \, \D t
= \int_{\al(B)} X_t \, \D t $ for all bounded Borel measurable $
B \subset \R^d $.

We need only coordinate permutations and shifts; that is,
homeomorphisms $\al$ of the form $ \al(p,s) : (t_1,\dots,t_d) \mapsto
( t_{p(1)}+s_1, \dots, t_{p(d)}+s_d ) $ where $ p
: \{1,\dots,d\} \to \{1,\dots,d\} $ is a permutation, and $ s \in \R^d
$. These $ \al(p,s) $ are a group of transformations. We denote by
$S_k$ (for $ k \in \{1,\dots,d\} $) the subgroup of all permutations
$p$ such that $ p(k+1)=k+1, \dots, p(d)=d $. We reformulate
Prop.~\ref{2.4} accordingly:
\begin{multline}
\text{if } X \text{ is \splittable{(1,k)}, then }
 X \circ \al(p,s) \text{ is \splittable{(1,k)}} \\
\text{ whenever } p \in S_k \text{ and } s \in \R^d \, .
\end{multline}

\begin{lemma}\label{2.5}
If $Y$ is a leak for $X$, then $ Y \circ \al(p,s) $ is a leak for $
X \circ \al(p,s) $ whenever $ p(1)=1 $ and $ s_1=0 $.
\end{lemma}

\begin{proof} \let\qed\relax
Having a split $(X^0,X^-,X^+)$ of $X$ such that\footnote{%
 Or rather, $Y$ is distributed like this.}
\[
Y_{t_1,\dots,t_d} = \begin{cases}
 X^-_{t_1,\dots,t_d} - X^0_{t_1,\dots,t_d} &\text{when } t_1<0,\\
 X^+_{t_1,\dots,t_d} - X^0_{t_1,\dots,t_d} &\text{when } t_1\ge0
\end{cases}
\]
and denoting $ \al = \al(p,s) $, we observe that $ ( X^0 \circ \al,
X^- \ccirc \al, X^+ \ccirc \al ) $ is a split of $ X \circ \al $, and
\[
(Y\circ\al)_{t_1,\dots,t_d} = \begin{cases}
 (X^-\ccirc\al)_{t_1,\dots,t_d} - (X^0\circ\al)_{t_1,\dots,t_d} &\text{when } t_1<0,\\
 (X^+\ccirc\al)_{t_1,\dots,t_d} - (X^0\circ\al)_{t_1,\dots,t_d}
  &\text{when } t_1\ge0.  \qquad\qquad \rlap{$\qedsymbol$}
\end{cases}
\]
\end{proof}

\begin{sloppypar}
\begin{proof}[Proof of Prop.~\ref{2.4}]
Induction in $k=1,\dots,d$. We inspect Def.~\ref{1.3}(a,b,c).

\begin{sloppypar}
(a) $ \int_{([0,1)^k+r)\times\R^{d-k}}
|X_{\al(p,s)(t)}| \, \D t = \int_{([0,1)^k+r')\times\R^{d-k}}
|X_t| \, \D t $ where \linebreak[4]
$ r' = (r_{p(1)}+s_1,\dots,r_{p(k)}+s_k) $.
\end{sloppypar}

(b) $ \int_A X_{\al(p,s)(t)} \, \D t = \int_B X_t \, \D t $
where $ B = \al(p,s)(A) $.

(c) We introduce transformations $ \be_{i,r} $ and $ \ga_k $ by\\
$ \be_{i,r} : (t_1,\dots,t_d) \mapsto
 (t_2,\dots,t_i,t_1+r,t_{i+1},\dots,t_d) $,\\
$ \ga_k : (t_1,\dots,t_d) \mapsto
 (t_k,t_1,\dots,t_{k-1},t_{k+1},\dots,t_d) $;\\
or, more formally, in terms of coordinate functions $ f_1,\dots,f_d
: \R^d \to \R $, $ f_i : (t_1,\dots,t_d) \mapsto t_i $, we have $
f_1 \circ \ga_k = f_k $, $ f_j \circ \ga_k = f_{j-1} $ for $j$ such
that $ 2 \le j \le k $ (if any), etc. We observe
in \ref{1.3}(c) that $ \ti X = X \circ \be_{i,r} $ and $ Y = \ti
 Y \circ \ga_k $. We have to check \ref{1.3}(c) for $ X \circ \al(p,s)
$. Given $ i \in \{1,\dots,k\} $ and $ r \in \R $, we note that
\begin{equation}\label{*}
f_i \circ \be_{i,r} = f_1+r \, ;
\end{equation}
taking $ i' = p^{-1}(i) $ we get
\[
f_{i'} \circ \al(p,s) \circ \be_{i,r} = f_{p(i')} \circ \be_{i,r} +
s_{i'} = f_i \circ \be_{i,r} + s_{i'} = f_1 + r + s_{i'} \, .
\]
We take $ r' = r+s_{i'} $ and get
\[
f_{i'} \circ \be_{i',r'} = f_1 + r' = f_1 + r + s_{i'} \, .
\]
By \eqref{*}, $ f_1 \circ \be_{i,r}^{-1} = f_i - r $ (since $ f_i
= f_i \circ \be_{i,r} \circ \be_{i,r}^{-1} = f_1 \circ \be_{i,r}^{-1}
+ r $), and similarly $ f_1 \circ \be_{i',r'}^{-1} = f_{i'} -
r' $. We introduce $ \de = \be^{-1}_{i',r'} \circ \al(p,s) \circ
\be_{i,r} $ and get
\begin{equation}\label{**}
f_1 \circ \de = f_1
\end{equation}
(since $ f_1 \circ \de = f_{i'} \circ \al(p,s) \circ \be_{i,r} - r' =
f_1 $), which shows that $ \de $ satisfies the condition of
Lemma \ref{2.5}.

Lemma \ref{2.5} applied to $\de$ and the leak $ \ti Y $ for $ \ti X =
X \circ \be_{i',r'} $ provides a leak $ \ti Y \circ \de $ for $ \ti
X \circ \de = X \circ \be_{i',r'} \circ \de =
X \circ \al(p,s) \circ \be_{i,r} $.

\textsc{Case $ k=1 $. } The permutation $p$ is necessarily trivial;
$ \al(p,s) $ is the shift by $s$; $ r'=r+s_1 $; $\de$ is the shift by
$ (0,s_2,\dots,s_d) $; $ Y = \ti Y $; and finally $ \int_{\R^d} |
(Y\circ\de)_t | \, \D t = \int_{\R^d} |Y_t| \, \D t $.

\textsc{Case $ k>1 $. } We know that $ \ti Y \circ \ga_k $
is \splittable{(1,k-1)}. We have $ f_1 \circ \ga_k = f_k $, thus
\[
f_k \circ \ga_k^{-1} = f_1
\]
(since $ f_1 = f_1 \circ \ga_k \circ \ga_k^{-1} = f_k \circ \ga_k^{-1}
$). We introduce $ \eps = \ga_k^{-1} \circ \de \circ \ga_k $ and note
that
\[
f_k \circ \eps = f_k
\]
(since $ f_k \circ \eps = f_k \circ \ga_k^{-1} \circ \de \circ \ga_k =
f_1 \circ \de \circ \ga_k = f_1 \circ \ga_k = f_k $); taking into
account that $ f_m \circ \eps = f_m $ for all $ m \in \{k+1,\dots,d\}
$ (since this relation holds for $\al(p,s)$, $ \ga_k $, $ \be_{i,r} $
and $ \be_{i',r'} $) we conclude that $ \eps $ is of the form
$ \al(p',s') $ with $ p' \in S_{k-1} $. By the induction
hypothesis, \splittability{(1,k-1)} of $ \ti Y \circ \ga_k $
implies \splittability{(1,k-1)} of $ \ti Y \circ \ga_k \circ \eps
= \ti Y \circ \de \circ \ga_k $, as required.
\end{proof}
\end{sloppypar}

\begin{corollary}\label{2.6}
If $ (X_{t_1,\dots,t_d})_{t_1,\dots,t_d\in\R} $ is \splittable{C}
and $ (\ell_1,\dots,\ell_d) $ is a permutation of $ (1,\dots,d) $,
then $ (X_{t_{\ell_1},\dots,t_{\ell_d}})_{t_1,\dots,t_d\in\R}
$ is \splittable{C}.
\end{corollary}

\begin{proof}
$ \( \frac1C X_{t_1,\dots,t_d} \)_{t_1,\dots,t_d\in\R} $ is
\splittable{(1,d)}; by Prop.~\ref{2.4}, $
\( \frac1C X_{t_{\ell_1},\dots,t_{\ell_d}} \)_{t_1,\dots,t_d\in\R} $
is \splittable{(1,d)}.
\end{proof}

\begin{corollary}\label{2.9}
If $ (X_t)_{t\in\R^d} $ is \splittable{C} and $ s \in \R^d $, then $
(X_{t+s})_{t\in\R^d} $ is \splittable{C}.
\end{corollary}

\begin{proof}
$ \frac1C X $ is \splittable{(1,d)}; by Prop.~\ref{2.4}, its shift
is \splittable{(1,d)}.
\end{proof}

\begin{remark}
If $ (X_t)_{t\in\R^d} $ is splittable, then $ (X_{ct})_{t\in\R^d} $ is
splittable whenever $ c \in (0,\infty) $. This is not evident because
of the unit cube used in Def.~\ref{1.3}(a); but see
Prop.~\ref{4.2}. Moreover, $ ( X_{c_1 t_1, \dots, c_d t_d}
)_{t_1,\dots,t_d\in\R} $ is splittable whenever $ c_1,\dots,c_d \in
(0,\infty) $. However, these facts will not be used here.
\end{remark}

Let $ d_1,d_2 \in \{1,2,\dots\} $, $ d=d_1+d_2 $, and $ B \subset
\R^{d_2} $ be a Borel measurable set. Let $X$ be a random field on $
\R^d $ such that for every bounded Borel measurable set $ A \subset
\R^{d_1} $ holds $ \int_{A\times B} |X_t| \, \D t < \infty $ a.s. We
introduce the random field $ X^{(B)} $ on $ \R^{d_1} $ by
\[
X_s^{(B)} = \int_B X_{s,t} \, \D t \quad \text{for } s \in \R^{d_1} \,
, \vspace{-5pt}
\]
that is,
\[
\int_A X_s^{(B)} \, \D s = \int_{A\times B} X_t \, \D t \quad
\text{for bounded Borel } A \subset \R^{d_1} \, .
\]

\begin{lemma}\label{2.10}
(a) If $ (X^0,X^-,X^+) $ is a split of $X$, then $ \( (X^0)^{(B)},
(X^-)^{(B)}, (X^+)^{(B)} \) $ is a split of $X^{(B)}$.

(b) If $Y$ is a leak for $X$, then $ Y^{(B)} $ is a leak for $ X^{(B)}
$.
\end{lemma}

\begin{proof}
(a) Immediate from Def.~{1.1}.

(b) Immediate from (a) and Def.~{1.2}.
\end{proof}

Let $X$ be a \splittable{(1,k)} random field on $ \R^d $, and $ k<d $,
and $ \ell \in \{1,\dots,d-k\} $. We treat $ \R^d $ as $ \R^{d-\ell}
\times \R^\ell $, consider $ X^{(B)} $ for $ B = \R^\ell $ and denote
it $ X^{(\R^\ell)} $ or just $ X^{(\ell)} $; this random field on $
\R^{d-\ell} $ is well-defined, since for bounded Borel $ A \subset
\R^{d-\ell} $ the relation ``$ \int_{A\times\R^{d-\ell}} |X_t| \, \D t
< \infty $ a.s.'' follows easily from Def.~{1.3}(a).

By Lemma \ref{2.10}(b), for each $ \ell \in \{1,\dots,d-k\} $,
\begin{equation}\label{2.10a}
\text{if } Y \text{ is a leak for } X \, , \quad \text{then }
Y^{(\ell)} \text{ is a leak for } X^{(\ell)} \, .
\end{equation}

\begin{lemma}\label{2.11}
If $X$ is \splittable{(1,k)} and $ 1 \le \ell \le d-k $, then $
X^{(\ell)} $ is a \splittable{(1,k)} random field on $ \R^{d-\ell} $.
\end{lemma}

\begin{sloppypar}
\begin{proof}
Induction in $k=1,\dots,d-\ell$. We inspect Def.~\ref{1.3}(a,b,c).

(a) $ \int_{([0,1)^k+s)\times\R^{d-\ell-k}} |X_t^{(\ell)}| \, \D t =
\int_{([0,1)^k+s)\times\R^{d-\ell-k}} \, \D t \big| \int_{\R^\ell}
X_{t,r} \, \D r \big| \le \int_{([0,1)^k+s)\times\R^{d-k}} |X_t| \, \D
t $, thus $ \Ex \exp(\dots) \le 2 $.

(b) $ \Ex X_s^{(\ell)} = \Ex  \int_{\R^\ell} X_{s,r} \, \D r =
\int_{\R^\ell} \Ex X_{s,r} \, \D r = 0 $.

(c) Let $ i \in \{1,\dots,k\} $ and $ r \in \R $; we consider the
random field $ \widetilde{X^{(\ell)}} $:
$ \widetilde{X^{(\ell)}}_{t_1,\dots,t_{d-\ell}} =
X^{(\ell)}_{t_2,\dots,t_i,t_1+r,t_{i+1},\dots,t_{d-\ell}}
= \int_{\R^\ell} X_{t_2,\dots,t_i,t_1+r,t_{i+1},\dots,t_d} \D
t_{d-\ell+1} \dots \D t_d = \int_{\R^\ell} \ti X_{t_1,\dots,t_d} \, \D
t_{d-\ell+1} \dots \D t_d = \ti X^{(\ell)}_{t_1,\dots,t_{d-\ell}} $;
that is, $ \widetilde{X^{(\ell)}} = \ti X^{(\ell)} $. We have a leak
$ \ti Y $ of $ \ti X $; by \eqref{2.10a}, $ \ti Y^{(\ell)} $ is a leak
of $ \ti X^{(\ell)} $.
Now, \splittability{(1,k-1)} of $Y$ implies \splittability{(1,k-1)} of
$Y^{(\ell)}$ by the induction hypothesis, provided that $k>1$; and $
Y^{(\ell)}_{t_1,\dots,t_{d-\ell}} = \int_{\R^\ell}Y_{t_1,\dots,t_d} \,
\D t_{d-\ell+1}\dots\D t_d = \int_{\R^\ell} \ti
Y_{t_k,t_1,\dots,t_{k-1},t_{k+1},\dots,t_d \, \D t_{d-\ell+1}} \dots
\D t_d = \ti
Y^{(\ell)}_{t_k,t_1,\dots,t_{k-1},t_{k+1},\dots,t_{d-\ell}} $, as
required. Otherwise, if $ k=1 $, we have $ \int_{\R^{d-\ell}}
| Y_t^{(\ell)} | \, \D t = \int_{\R^{d-\ell}} \D t \big|
\int_{\R^\ell} Y_{t,r} \, \D r \big| \le \int_{\R^d} |Y_t| \, \D t $,
thus $ \Ex \exp (\cdot) \le 2 $.
\end{proof}
\end{sloppypar}

\begin{corollary}\label{2.12}
If $X$ is \splittable{(C,k)} on $ \R^d $ and $ k < d $, then $
X^{(d-k)} $ is \splittable{C} on $ \R^k $.
\end{corollary}

\begin{proof}
Lemma \ref{2.11} for $ \ell = d-k $ gives \splittability{(1,k)} of $
\frac1C X^{(d-k)} $ on $ \R^k $.
\end{proof}

This fact generalizes readily, as follows.

\begin{corollary}\label{2.12a}
Let $X$ be \splittable{(C,k)} on $ \R^d $, $ k < d $, and $ B \subset
\R^{d-k} $ a Borel measurable set. Then $ X^{(B)} $ is \splittable{C}
on $ \R^k $.
\end{corollary}

\begin{proof}\!\footnote{%
 $ \One_A $ stands for the indicator function of a set $A$; it equals
 $1$ within $A$ and $0$ outside $A$.}
$ X \cdot \One_{\R^k \times B} $ is \splittable{(C,k)} on $ \R^d $ by
Prop.~\ref{2.1}, and $ X^{(B)} = ( X \cdot \One_{\R^k \times B}
)^{(d-k)} $.
\end{proof}

\begin{lemma}\label{2.12b}\!\footnote{%
 To be used in Sect.~\ref{sect5}.}
Let $X$ be \splittable{C} on $ \R^d $. Then $ X $ has a leak $ Y $
such that the random field $ Z $ on $ \R^{d-1} $ defined by $
Z_{t_1,\dots,t_{d-1}} = \int_{\R} Y_{s,t_1,\dots,t_{d-1}} \, \D s
$ is \splittable{C} (on $ \R^{d-1} $).
\end{lemma}

\begin{proof}
WLOG, $ C=1 $ (otherwise divide $ X, Y, Z $ by
$C$). Def.~\ref{1.3}(c) for $ k=d $, $ i=1 $ and $ r=0 $ gives a leak
$ \ti Y $ for $ \ti X = X $ such that the random field $ Y $ defined
by $ Y_{t_1,\dots,t_d} = \ti Y_{t_d,t_1,\dots,t_{d-1}} $ is
\splittable{(1,d-1)}. We have $ Z_{t_1,\dots,t_{d-1}} = \int_{\R} \ti
Y_{s,t_1,\dots,t_{d-1}} \, \D s = \int_{\R}
Y_{t_1,\dots,t_{d-1},s} \, \D s $, that is, $ Z = Y^{(1)} $. We
apply to $Y$ Corollary \ref{2.12} with $ k = d-1 $ (and $ C=1 $).
\end{proof}

\begin{remark}\label{2.12c}
Once again, the same holds for $Z$ defined by $
Z_{t_1,\dots,t_{d-1}} = \int_B Y_{s,t_1,\dots,t_{d-1}} \, \D s $
whenever $ B \subset \R $ is a Borel measurable set. (Just use
Corollary \ref{2.12a} rather than \ref{2.12}; now $ Z = Y^{(B)} $.)
\end{remark}

\begin{proposition}\label{2.13}
Let $X$ be a \splittable{C} random field on $ \R^d $, $ d>1 $, $ B
\subset \R^{d-1} $ a box, and $ r>0 $. Then there exist random
variables $ U,V,W,Z $ (on some probability space) and a \splittable{C}
random field $Y$ on $ \R^{d-1} $ such that $U$, $V$ are (mutually)
independent, $ W + Z = U + V $, and
\[
\begin{aligned}
U &\sim \int_{[-r,0)\times B} X_t \, \D t \, , \qquad \text{(here ``$\sim$''
  means ``distributed as'')} \\
V &\sim \int_{[0,r)\times B} X_t \, \D t \, , \\
W &\sim \int_{[-r,r)\times B} X_t \, \D t \, , \\
Z &\sim \int_B Y_s \, \D s \, .
\end{aligned}
\]
\end{proposition}

\begin{sloppypar}
\begin{proof}
WLOG, $ X_{t_1,\dots,t_d} = 0 $ whenever $ t_1 \notin [-r,r) $; otherwise we
multiply $X$ by $ \One_{[-r,r)\times\R^{d-1}} $ and use Corollary
\ref{2.3}. By Item (c) of Def.~\ref{1.3} for $ k=d $, $ i=1 $ and $
r=0 $, $ \frac1C X $ has a leak $ \frac1C \ti Y $ such that the random field $ (\ti
Y_{t_d,t_1,\dots,t_{d-1}})_{t_1,\dots,t_d\in\R} $ is
\splittable{(C,d-1)}. By Corollary \ref{2.12} for $ k=d-1 $, the
random field $ Y = ( \int_{\R} \ti Y_{s,t_1,\dots,t_{d-1}} \, \D s
)_{t_1,\dots,t_{d-1}\in\R} = ( \int_{\R} \ti Y_{t_1,\dots,t_d} \, \D
t_1 )_{t_2,\dots,t_d\in\R} $ is \splittable{C} on $ \R^{d-1} $. We
take a split $(X^0,X^-,X^+)$ of $X$ whose leak is $\ti Y$, and
introduce
\vspace{-1pt}
\[
\begin{aligned}
U &= \int_{[-r,0)\times B} X_t^- \, \D t \, , \\
V &= \int_{[0,r)\times B} X_t^+ \, \D t \, , \\
W &= \int_{[-r,r)\times B} X_t^0 \, \D t \, , \\
Z &= \int_{[-r,r)\times B} \ti Y_t \, \D t = \int_B Y_s \, \D s \, .
\end{aligned}
\vspace{-1pt}
\]
Independence of $ X^-, X^+ $ implies independence of $ U,V $.
The equalities $ \int_{[-r,0)\times B} ( X_t^- - X_t^0 ) \, \D t =
\int_{[-r,0)\times B} \ti Y_t \, \D t $ and $ \int_{[0,r)\times B} (
X_t^+ - X_t^0 ) \, \D t = \int_{[0,r)\times B} \ti Y_t \, \D t $ imply
$ U+V-W=Z $.
\end{proof}
\end{sloppypar}

\begin{remark}\label{2.19}
For $ d=1 $ we have no $B$, no $Y$, but $ U \sim \int_{[-r,0)} X_t \,
\D t $, $ V \sim \int_{[0,r)} X_t \, \D t $, $ W \sim \int_{[-r,r)}
X_t \, \D t $, and $ \Ex \exp |Z| \le 2 $, and $ \Ex Z = 0 $.

In the proof: $Y$ is just a random variable $ \int_\R \ti Y_t \, \D t
$, and $ \Ex \exp \int_{\R} |\ti Y_t| \, \D t \le 2 $; $ Z = Y $; $
|Z| \le \int_{\R} | \ti Y_t | \, \D t $; thus $ \Ex \exp |Z| \le 2 $.
\end{remark}

\section[Upper bounds: abstract nonsense]
  {\raggedright Upper bounds: abstract nonsense}
\label{sect3}
We consider functions $ f : \cup_{d=0}^\infty \( \R \times (0,+\infty)^d \) \to
[0,+\infty] $, where $ \R \times (0,+\infty)^0 $ means just $ \R
$.\,\footnote{%
 The sets $ \R \times (0,+\infty)^d $ are pairwise disjoint, of course.}

Informally, $ f(\la; r_1,\dots,r_d) $ is intended to be an upper bound
on the cumulant generating function
\[
\log \Ex \exp \frac{\la}{\sqrt{r_1\dots
r_d}} \int_{[0,r_1)\times\dots\times[0,r_d)} X_t \, \D t
\]
for a class of random fields $ (X_t)_{t\in\R^d} $.
Invariance (in distribution) under permutations (of coordinates $
t_1,\dots,t_d $) and shifts, not required of a random field, is
required of the class of random fields, which motivates the first
condition on $f$:
\begin{multline}\label{3.1}
f(\la; r_1,\dots,r_d) = f(\la; r_{k_1},\dots,r_{k_d}) \\
\text{whenever } (k_1,\dots,k_d) \text{ is a permutation of } (1,\dots,d) \, .
\end{multline}
The second condition on $f$ (``duplication inequality'') is more
complicated; for its motivation see Prop.~\ref{4.1aa}:
\begin{equation}\label{3.2}
f(\la; r_1,\dots,r_d) \le \frac2p
f \Big( \frac{p\la}{\sqrt2}; \frac{r_1}2, r_2, \dots, r_d \Big)
+ \frac{p-1}p f \Big( -\frac{p}{p-1} \frac{\la}{\sqrt{r_1}};
r_2,\dots,r_d \Big)
\end{equation}
whenever $ p \in (1,\infty) $.

Here is the third (last) condition on $f$; for its motivation
see Prop.~\ref{4.2}:
\begin{multline}\label{3.3}
\forall d \in \{1,2,\dots\} \;\; \forall C \in [1,\infty) \;\; \exists \eps>0 \\
\forall r_1,\dots,r_d \in [\tfrac12 C, C] \;\;
  \forall \la \in [-\eps,\eps] \quad \eps^2 f(\la; r_1,\dots,r_d) \le \la^2 \, .
\end{multline}

More formally, for each $ d \in \{0,1,2,\dots\} $ we introduce
condition $ \A_d $ as follows. For $ d \ne 0 $ condition $ \A_d $ is
the conjunction of \eqref{3.1}, \eqref{3.2} and \eqref{3.3} (for the
given $d$). And condition $ \A_0 $ is just
\begin{equation}\label{3.4}
\exists \eps>0 \;\; \forall \la \in [-\eps,\eps] \quad \eps^2 f(\la) \le \la^2 \, .
\end{equation}

\begin{theorem}\label{II-1.5}
If $f$ satisfies condition $ \A_d $ for all $d$, then for each $d$
there exists $C_d$ such that $ f(\la;r_1,\dots,r_d) \le C_d \la^2 $
whenever $ C_d |\la| \le \sqrt{r_1\dots r_d} \log^{-d} (r_1\dots r_d)
$ and $ \min(r_1,\dots,r_d) \ge C_d $.\,\footnote{%
 $ \, \log^{-d} (\dots) $ means $ 1/\(\log(\dots))^d $, of course.}
\end{theorem}

This is basically Theorem 1.5 of \cite{II}, somewhat reformulated and
generalized. For convenience, for each $ d \in \{0,1,2,\dots\} $ we
introduce condition $ \B_d $ as follows.
For $ d \ne 0 $, condition $ \B_d $ requires existence of $ C_d $ such
that $ f(\la;r_1,\dots,r_d) \le C_d \la^2 $ whenever $ C_d
|\la| \le \sqrt{r_1\dots r_d} \log^{-d} (r_1\dots r_d) $ and
$ \min(r_1,\dots,r_d) \ge C_d $. For $ d=0 $, condition $ \B_0 $
requires existence of $ C_0 $ such that $ f(\la) \le C_0 \la^2 $
whenever $ C_0 |\la| \le 1 $. Note that $ \B_0 \equi \A_0 $.

Thus, Theorem \ref{II-1.5} says that
$ ( \forall d \; \A_d ) \impl ( \forall d \; \B_d ) $.
Induction in $ d $, started with the trivial relation
$ \A_0 \impl \B_0 $, reduces Theorem \ref{II-1.5} to the following
result (to be proved later).

\begin{proposition}\label{3.6}
For each $ d \in \{1,2,\dots\} $ holds $ ( \A_d \land \B_{d-1}
) \impl \B_d $.
\end{proposition}

It is possible to translate nearly all calculations in \cite{II} from
the language of $ f_B(\la) $ used there into the language of $
f(\la;r_1,\dots,r_d) $ used here. Alternatively, we define $ f_B(\la)
$ in our terms:
\begin{equation}\label{3.7}
f_B(\la) = f(\la;\be_1-\al_1,\dots,\be_d-\al_d) \quad \text{for } B =
[\al_1,\be_1] \times \dots \times [\al_d,\be_d] \, .
\end{equation}

We rewrite Lemma 2.4 of \cite{II} in our terms and prove it in our
framework. For simplicity we formulate it as a duplication inequality
for the first coordinate, but it generalizes readily to $k$-th
coordinate ($k=1,\dots,d$) due to \eqref{3.1}.

\begin{lemma}\label{3.8}
Let $f$ satisfy conditions $ \A_d $ and $ \B_{d-1} $ for a given $
d \ge 1 $; $ r_1 \in (0,\infty) $; $ r_2,\dots,r_d \in
[C_{d-1},\infty) $. Then for all $ p \in (1,\infty) $ and $ \la \in \R
$ satisfying $ C_{d-1}
|\la| \le \frac{p-1}{p} \sqrt{2v} \log^{-(d-1)} \frac{v}{r_1} $, where
$ v = r_1 \dots r_d $ (for $ d=1 $ read $ C_0
|\la| \le \frac{p-1}{p} \sqrt{2r_1} \, $), holds
\[
f(\la;2r_1,r_2,\dots,r_d) \le \frac2p f \Big( \frac{p\la}{\sqrt2};
r_1,\dots,r_d \Big) + C_{d-1} \frac{p}{p-1} \frac{\la^2}{2r_1} \, .
\]
\end{lemma}

\begin{sloppypar}
\begin{proof}
By \eqref{3.2}, $ f(\la;2r_1,r_2,\dots,r_d) \le \frac2p f \Big( \frac{p\la}{\sqrt2};
r_1,\dots,r_d \Big) + \frac{p-1}{p} f \(
- \frac{p}{p-1} \frac{\la}{\sqrt{2r_1}};r_2,\dots,r_d \) $;
it suffices to prove that $ f(-\mu;r_2,\dots,r_d) \le
C_{d-1} \mu^2 $ where $ \mu
= \frac{p}{p-1} \frac{\la}{\sqrt{2r_1}} $.
Condition $ \B_{d-1} $
reduces it to $ C_{d-1} |\mu| \le \sqrt{r_2\dots r_d} \log^{-(d-1)}
(r_2\dots r_d) $, that is,
$ C_{d-1}
|\mu| \le \sqrt{ \frac{v}{r_1} } \log^{-(d-1)} \frac{v}{r_1} $ (for $
d=1 $ read $ C_0|\mu| \le 1 $); it
remains to rewrite the latter in terms of $ \la $.
\end{proof}
\end{sloppypar}

This lemma is the only property of $ f_B(\la) $ used in the proof
of \cite[Prop.~2.6]{II}. The same proof applies to $ f_B(\la) $
defined by \eqref{3.7}, which gives Prop.~\ref{II-2.6} below.
Similarly to \cite[p.~5]{II} we denote for convenience
\[
R(v) = v^{\frac1d} \quad \text{and} \quad S(v) =
v^{\frac{d-1}d} \qquad \text{for all } v \in (0,\infty) \, .
\]

\begin{proposition}\label{II-2.6}
Let $f$ satisfy conditions $ \A_d $ and $ \B_{d-1} $ for a given $
d \ge 1 $, and $C$ be large enough. Then for every $ \de>0 $ there
exists natural $N$ such that for all $ r_1,\dots,r_d \in
[C,2C] $ and $ a \ge \frac{C}{R(r_1\dots r_d)} $ satisfying
\[
f_{B_0}(\la) \le a \la^2 \quad \text{for all } \la \in
[-\de,\de] \, ,
\]
where $ B_0 = [0,r_1]\times\dots\times[0,r_d] $ (and $ f_{B_0}(\la) $
is defined by \eqref{3.7}),  the following holds for all $
n_1,\dots,n_d \in \{0,1,2.\dots\} $ satisfying $ n_1+\dots+n_d \ge N
$:
\[
f_B(\la) \le 2a \la^2 \quad \text{for all } \la \in [-\De,\De] \, ,
\]
where $ B = [0,2^{n_1} r_1]\times\dots\times[0,2^{n_d} r_d] $ and
$ \De = \frac1{C_{d-1}} \sqrt{ \frac1a S ( \vol B ) } \log^{-(d-1)} S
( \vol B ) $.
\end{proposition}

(For $d=1$, by convention, $ \log^{-(d-1)} S(\dots) = 1 $,
notwithstanding that $ S(\dots)=1 $.)

The formulation above is stronger than \cite[Prop.~2.6]{II} because it
incorporates the fact that $N$ depends only on $ d $, $ C $ and $ \de
$ noted in \cite[Remark~2.7]{II}.
For completeness, here is the same result in terms of $
f(\la;r_1,\dots,r_d) $.

\begin{customprop}{\ref*{II-2.6}a}\label{3.9a}
Let $f$ satisfy conditions $ \A_d $ and $ \B_{d-1} $ for a given $
d \ge 1 $, and $C$ be large enough. Then for every $ \de>0 $ there
exists natural $N$ such that the following holds for all $
r_1,\dots,r_d \in [C,2C] $, $ a \ge \frac{C}{R(r_1\dots r_d)} $, $
n_1,\dots,n_d$  satisfying $ n_1+\dots+n_d \ge N $, provided that
$ \forall \la \in [-\de,\de] \;\> f(\la;r_1,\dots,r_d) \le a \la^2 $:
\[
f(\la; 2^{n_1} r_1, \dots, 2^{n_d} r_d) \le
2a \la^2 \quad \text{for all } \la \in [-\De,\De] \, ,
\]
where $ \De = \frac1{C_{d-1}} \sqrt{ \frac1a S ( 2^{n_1} r_1 \dots
2^{n_d} r_d ) } \log^{-(d-1)} S ( 2^{n_1} r_1 \dots 2^{n_d} r_d )
$.
\end{customprop}

We define for convenience, as in \cite[Sect.~2]{II},
\[
f_{v,C} (\la) = \sup_{r_1\dots r_d=v, \min(r_1,\dots,r_d)\ge C} f(\la;
r_1,\dots,r_d)
\]
for $ v \ge C^d $.

Now we adapt Prop.~2.3 of \cite{II}.

\begin{proposition}\label{II-2.4}
Let $f$ satisfy conditions $ \A_d $ and $ \B_{d-1} $ for a given $
d \ge 1 $. Then there exists $ C \in (1,\infty) $ such that
\[
f_{v,C} (\la) \le C \la^2 \quad \text{whenever} \quad C
|\la| \le \sqrt{S(v)} \log^{-(d-1)} v \text{ and } v \ge C^d \, .
\]
\end{proposition}

\begin{sloppypar}
\begin{proof}
We take $C$ large enough according to
Prop.~\ref{3.9a}. Condition \eqref{3.3} gives $ \eps>0 $ such that
$ \eps^2 f(\la; r_1,\dots,r_d) \le \la^2 $ whenever $ r_1,\dots,r_d \in
[C, 2C] $ and $ |\la| \le \eps $. We take $ \de = \eps $, $ a
= \max \( 1, \frac1{\eps^2} \) $. Prop.~\ref{3.9a} gives $N$. We take $ M
= \max ( 2a, 2C \cdot 2^{N/d} ) $. It is sufficient to prove that $
f_{v,M} (\la) \le M \la^2 $ whenever $
M|\la| \le \sqrt{S(v)} \log^{-(d-1)} v $ and $ v \ge M^d $. That is, $
f(\la;r_1,\dots,r_d) \le M \la^2 $ whenever $ r_1\dots r_d = v $, $
r_1,\dots,r_d \ge M $.

Given such $ r_1,\dots,r_d $, we have $ r_1,\dots, r_d \ge M \ge C $;
Lemma \cite[3.1]{III} gives $ n_1,\dots,n_d $ such that $ 2^{-n_1}
r_1, \dots, 2^{-n_d} r_d \in [C,2C) $ and $ 2^{-d} r_1\dots r_d < C^d
2^{n_1+\dots+n_d} $, whence $ 2^{n_1+\dots+n_d} > (2C)^{-d} r_1\dots
r_d \ge (2C)^{-d} M^d \ge 2^N $ and so, $ n_1+\dots+n_d > N $.

Taking into account that $ f(\la; 2^{-n_1} r_1, \dots, 2^{-n_d} r_d
) \le \frac1{\eps^2} \la^2 \le a \la^2 $ for $ \la \in [-\de,\de] \subset
[-\eps,\eps] $, and $ a \ge 1 \ge \frac{C}{R(2^{-n_1} r_1 \dots
2^{-n_d})} $, we apply Prop.~\ref{3.9a} to $ 2^{-n_1} r_1, \dots,
2^{-n_d} r_d $, getting $ f(\la;r_1,\dots,r_d) \le 2a \la^2 \le
M \la^2 $ as needed.
\end{proof}
\end{sloppypar}

Toward proving Prop.~\ref{3.6} we assume $ \A_d $ and $ \B_{d-1} $ for
a given $ d \ge 1 $; we need to deduce $ \B_d $. Prop.~\ref{II-2.4}
gives a constant $ C $; denote it $ C_{\text{\ref*{II-2.4}}} $. WLOG,
$ C_{\text{\ref*{II-2.4}}} \ge C_{d-1} $. We define for convenience,
as in \cite[Sect.~3]{II},
\[
f_v(\la) = f_{v,C_{\text{\ref*{II-2.4}}}} (\la) = \sup_{r_1\dots
r_d=v, \min(r_1,\dots,r_d)\ge C_{\text{\ref*{II-2.4}}}} f(\la;
r_1,\dots,r_d) \quad \text{for } v \ge C_{\text{\ref*{II-2.4}}}^d \, ,
\]
and adapt Lemma 3.2 of \cite{II}.

\begin{lemma}
For all $ p \in (1,\infty) $,
\[
f_{2v}(\la) \le \frac2p f_v \Big( \frac{p\la}{\sqrt2} \Big) +
C_{d-1} \frac{p}{p-1} \cdot \frac{\la^2}{R(2v)}
\]
whenever $ C_{d-1} |\la| \le \frac{p-1}{p} \sqrt{2v} \log^{-(d-1)}
S(2v) $ and $ 2v \ge (2C_{\text{\ref*{II-2.4}}})^d $.
\end{lemma}

\begin{proof}
Given $ r_1,\dots,r_d \in [C_{\text{\ref*{II-2.4}}},\infty) $ such that
$ r_1\dots r_d = 2v $, we assume that $ \max(r_1,\dots,r_d) = r_1 $
(WLOG, due to \eqref{3.1}) and apply Lemma \ref{3.8} to $ \frac12 r_1,
r_2, \dots, r_d $:
\[
f(\la; r_1,\dots,r_d) \le \frac2p
f \Big( \frac{p\la}{\sqrt2}; \frac{r_1}2,r_2,\dots,r_d \Big) +
C_{d-1} \frac{p}{p-1} \frac{\la^2}{r_1}
\]
for $ C_{d-1} |\la| \le \frac{p-1}p \sqrt{2v} \log^{-(d-1)} \frac{2v}{r_1} $. We 
note that $ r_1 \ge R(2v) \ge R\( (2C_{\text{\ref*{II-2.4}}})^d \) = 2
C_{\text{\ref*{II-2.4}}} $, thus $ f \( \frac{p\la}{\sqrt2};
\frac{r_1}2,r_2,\dots,r_d \) \le f_v\( \frac{p\la}{\sqrt2} \) $ and
$ \frac{\la^2}{r_1} \le \frac{\la^2}{R(2v)} $. Also,
$ \frac{2v}{r_1} \le \frac{2v}{R(2v)} = S(2v) $, thus $ \log^{-(d-1)}
S(2v) \le \log^{-(d-1)}\frac{2v}{r_1} $.
\end{proof}

This lemma is the only property of $ f_v(\la) $ used in \cite{II} when
proving Corollaries 3.3, 3.4, 3.5, Lemmas 3.9, 3.11, 3.13, 3.14, 3.16
and ultimately Prop.~3.6 ($ C_1 $ there is $ C_{d-1} $ here; $
C_2 $ there is $ C_{\text{\ref*{II-2.4}}} $ here; and (3.1) there
is ensured by Prop.~\ref{II-2.4} here), which gives the
following remake of \cite[Prop.~3.6]{II}.

\begin{proposition}
There exists $ C \in (1,\infty) $ such that for every $ \la \in \R $,
\[
\text{if } \; C|\la| \le \sqrt v \log^{-d} v \; \text{ and } \; v \ge
C^d \, , \quad \text{then} \quad f_{v,C}(\la) \le C \la^2 \, .
\]
\end{proposition}

Denoting this $C$ by $C_{d}$ we get condition $ \B_d $, thus deduced
from $ \A_d $ and $ \B_{d-1} $. This completes the proof of
Prop.~\ref{3.6} and Theorem \ref{II-1.5}.

\section[Upper bounds, applied]
  {\raggedright Upper bounds, applied}
\label{sect4}
We define for $ d \in \{1,2,\dots\} $
\begin{equation}\label{4.1}
f(\la; r_1,\dots,r_d) = \sup_X \log \Ex \exp \frac{\la}{\sqrt{r_1\dots
r_d}} \int_{[0,r_1)\times\dots\times[0,r_d)} X_t \, \D t
\end{equation}
where the supremum is taken over all \splittable{1} random fields $X$ on
$ \R^d $; and for $d=0$,
\begin{equation}\label{4.1aaa}
f(\la) = \begin{cases}
  \la^2 &\text{for } |\la| \le 1,\\
  +\infty &\text{for } |\la| > 1,
\end{cases}
\end{equation}

Condition \eqref{3.1} is satisfied due to Corollary \ref{2.6}.

\begin{proposition}\label{4.1aa}
Condition \eqref{3.2} is satisfied.
\end{proposition}

We borrow a general fact \cite[Lemma 1.7]{2008}.

\begin{lemma}\label{4.1ab}
For all random variables $ X,Y $ and all $ p \in (1,\infty) $,
\begin{multline*}
p \log \Ex \exp \frac1p X - (p-1) \log \Ex \exp \Big( - \frac1{p-1} Y \Big) \le
 \\
\le \log \Ex \exp (X+Y) \le \frac1p \log \Ex \exp pX + \frac{p-1}p \log \Ex \exp
\frac{p}{p-1} Y \, .
\end{multline*}
(In the lower bound we interpret $ \infty-\infty $ as $ -\infty $.)
\end{lemma}

\begin{proof}
By the H\"older inequality,
\[
\Ex \exp(X+Y) = \Ex ( \exp X \cdot \exp Y ) \le ( \Ex \exp pX )^{1/p} \Big(
\Ex \exp \frac{p}{p-1} Y \Big)^{(p-1)/p} \, ;
\]
the upper bound follows. We apply the upper bound to $ \frac1p (X+Y) $ and $
\( -\frac1p Y \) $ instead of $ X,Y $:
\[
\log \Ex \exp \frac1p X \le \frac1p \log \Ex \exp (X+Y) + \frac{p-1}p \log \Ex
\exp \Big( -\frac1{p-1} Y \Big) \, ;
\]
the lower bound follows.
\end{proof}

\begin{lemma}\label{4.1a}
Let $X$ be a \splittable{1} random field on $ \R^d $, $d>1$, $
r,r_2,\dots,r_d>0 $, and $ B = [0,r_2)\times\dots\times[0,r_d) \subset
\R^{d-1} $. Then for all $ \la \in \R $ holds\footnote{%
 The lower bound will be used in Sect.~\ref{sect5}.}
\begin{multline*}
p \log \Ex \exp \frac\la{p} \int_{[-r,0)\times B} X_t \, \D t +
 p \log \Ex \exp \frac\la{p} \int_{[0,r)\times B} X_t \, \D t \, - \\
- (p-1) f \Big( \frac\la{p-1} \sqrt{\vol B}; r_2,\dots,r_d \Big) \le \\
\le \log \Ex \exp \la \int_{[-r,r)\times B} X_t \, \D t \le
 \frac1p \log \Ex \exp p\la \int_{[-r,0)\times B} X_t \, \D t \, + \\
+ \frac1p \log \Ex \exp p\la \int_{[0,r)\times B} X_t \, \D t +
 \frac{p-1}p f \Big( -\frac{p}{p-1} \la\sqrt{\vol B}; r_2,\dots,r_d
\Big) \, .
\end{multline*}
\end{lemma}

\begin{proof}
Prop.~\ref{2.13} gives a \splittable{1} random field $Y$ on $ \R^{d-1}
$ and random variables $ U,V,W,Z $ such that $U,V$
are independent, $W+Z=U+V$, and
$ U \sim \int_{[-r,0)\times B} X_t \, \D t $,
$ V \sim \int_{[0,r)\times B} X_t \, \D t $,
$ W \sim \int_{[-r,r)\times B} X_t \, \D t $,
$ Z \sim \int_B Y_s \, \D s $.
It remains to apply Lemma \ref{4.1ab} to $ X = \la(U+V) $, $ Y = -\la
Z $ and use independence of $U,V$.
\end{proof}

Another general fact, borrowed from \cite[Lemma 2a8]{I} (see also
\cite[Lemma 1.8]{II}: for arbitrary random variable $Z$,
\begin{equation}\label{***}
\text{if } \Ex \exp |Z| \le 2 \text{ and } \Ex Z = 0 \, , \;\>
\text{then} \;\; \forall \la \in [-1,1] \;\> \log \Ex \exp \la Z \le
\la^2 \, .
\end{equation}

\begin{proof}[Proof of Prop.~\ref{4.1aa}]
For $ d>1 $, Lemma \ref{4.1a} (the upper bound) applied to $ r = \frac12 r_1 $ and $
\frac\la{\sqrt{r_1\dots r_d}} $ instead of $\la$ gives
\begin{multline*}
\log \Ex \exp \frac{\la}{\sqrt{r_1\dots r_d}} \int_{[-r,r)\times B}
 X_t \, \D t \le
 \frac1p \log \Ex \exp \frac{p\la}{\sqrt{r_1\dots r_d}}
 \int_{[-r,0)\times B} X_t \, \D t \, + \\
+ \frac1p \log \Ex \exp \frac{p\la}{\sqrt{r_1\dots r_d}}
 \int_{[0,r)\times B} X_t \, \D t +
 \frac{p-1}p f \Big( -\frac{p}{p-1} \frac\la{\sqrt{r_1}}; r_2,\dots,r_d
 \Big) \, ;
\end{multline*}
the first term (in the right-hand side) does not exceed $ \frac1p f \Big( \frac{p\la}{\sqrt 2};
\tfrac12 r_1, r_2,\dots,r_d \Big) $; the same holds for the second
term. It remains to take the supremum over all \splittable{1} random fields
$X$ on $\R^d$.

For $ d=1 $, Remark~\ref{2.19} gives ($U,V,W$ and) $Z$ such that $
\Ex \exp |Z| \le 2 $ and $ \Ex Z = 0 $. By \eqref{4.1aaa} and
\eqref{***}, $ \log \Ex \exp \la Z \le f(\la) $ for all $ \la $ (and all
relevant $Z$), and so, $ \log \Ex \exp \frac{\la}{\sqrt{r}} W \le
\frac2p f \( \frac{p\la}{\sqrt 2}; \tfrac12 r \) + \frac{p-1}p f \(
-\frac{p}{p-1} \frac{\la}{\sqrt{r}} \) $.
\end{proof}

\begin{proposition}\label{4.2}
Condition \eqref{3.3} is satisfied.
\end{proposition}

\begin{sloppypar}
We introduce $ \phi : (0,\infty)^d \to [0,\infty] $ by
\[
\phi(r_1,\dots,r_d) = \limsup_{\la\to0} f(\la;r_1,\dots,r_d)
\]
and prove that this function is identically zero. Note that $
\phi(r_1,\dots,r_d) = \limsup_{\la\to0} f(\la\sqrt{r_1,\dots,r_d};
r_1,\dots,r_d) $, and $ f(\la\sqrt{r_1,\dots,r_d}; r_1,\dots,r_d) = 
\sup_X \log \Ex \exp \la \int_{[0,r_1)\times\dots\times[0,r_d)} X_t \,
\D t = \log \sup_X \Ex \exp \la \int_{[0,r_1)\times\dots\times[0,r_d)}
X_t \, \D t $.
\end{sloppypar}

\begin{lemma}\label{4.3}
The function $ \phi $ is increasing.
\end{lemma}

\begin{proof}
Given $ r_1 \le s_1, \dots, r_d \le s_d $ and a \splittable{1} $X$, we
note that $ Y = X \cdot \One_{[0,r_1)\times\dots\times[0,r_d)} $ is
\splittable{1} by Corollary \ref{2.3}, and $
\int_{[0,s_1)\times\dots\times[0,s_d)} Y_t \,
\D t = \int_{[0,r_1)\times\dots\times[0,r_d)} X_t \, \D t $, whence $
\log \Ex \exp \la \int_{[0,r_1)\times\dots\times[0,r_d)} X_t \, \D t=\\
\log \Ex \exp \la \int_{[0,s_1)\times\dots\times[0,s_d)} Y_t \, \D t
\le f(\la\sqrt{s_1,\dots,s_d}; s_1,\dots,s_d) $; supremum in $X$ gives
$ f(\la\sqrt{r_1,\dots,r_d}; r_1,\dots,r_d) \le
f(\la\sqrt{s_1,\dots,s_d}; s_1,\dots,s_d) $; $ \limsup $ in $\la$ gives
$ \phi(r_1,\dots,r_d) \le \phi(s_1,\dots,s_d) $.
\end{proof}

\begin{sloppypar}
\begin{lemma}\label{4.4}
For arbitrary $ r_2,\dots,r_d \in (0,\infty) $ the function $ \psi : r
\mapsto \phi(r, r_2,\dots,r_d) $ satisfies $ \psi(r+s) \le
\( \psi(r)+\psi(s) \) / 2 $ for all $ r,s \in (0,\infty) $.
\end{lemma}
\end{sloppypar}

\begin{proof}
Denoting $ B = [0,r_2)\times\dots\times[0,r_d) \subset \R^{d-1} $ we have\\
$ \Ex \exp \la \int_{[0,r+s)\times B} X_t \, \D t = \Ex \exp
\( \la \int_{[0,r)\times B} X_t \, \D t + \la \int_{[r,r+s)\times B}
X_t \, \D t \) \le\\
\( \Ex \exp 2\la \int_{[0,r)\times B} X_t
\, \D t \)^{1/2} \( \Ex \exp 2\la \int_{[r,r+s)\times B} X_t \, \D
t \)^{1/2} \le \\
\exp \frac12 \( f(2\la; r,r_2,\dots,r_d) + f(2\la;
s,r_2,\dots,r_d) \) $; supremum in $X$ gives\\
$ f(\la;
r+s,r_2,\dots,r_d) \le \frac12 \( f(2\la; r,r_2,\dots,r_d) + f(2\la;
s,r_2,\dots,r_d) \) $, and then\\
$ \limsup $ in $\la$ gives $ \psi(r+s) \le \frac12 \(
\psi(r) + \psi(s) \) $.
\end{proof}

\begin{lemma}\label{4.5}
The function $ \phi $ is constant.
\end{lemma}

\begin{proof}
We'll prove that the function $ \psi $ (introduced in Lemma \ref{4.4})
is constant; this is sufficient due to \eqref{3.1}. For every $ r \in
(0,\infty) $, first, $ \psi(s) \le \psi(r) $ for all $ s \in (0,r] $
by Lemma \ref{4.3}; second, $ \psi(s) \le \psi(r) $ for all $ s \in
(0,2r] $ by Lemma \ref{4.4}; and so on; $ \psi(s) \le \psi(r) $ for
all $ s \in (0,\infty) $ and all $ r \in (0,\infty) $.
\end{proof}

\begin{lemma}\label{4.6}
$ \phi(r_1,\dots,r_d) = 0 $ for all $ r_1,\dots,r_d \in (0,\infty) $.
\end{lemma}

\begin{proof}
By Lemma \ref{4.5} it is sufficient to prove that $ \phi(1,\dots,1) =
0 $. By Def.~\ref{1.3}(a) (for $k=d$), $ \Ex \exp \int_{[0,1)^d} |X_t|
\, \D t \le 2 $ for every \splittable{1} $X$. We apply \eqref{***},
taking into account that $ \Ex \int_{[0,1)^d} X_t \, \D t = 0
$ by \ref{1.3}(b) and $ | \int_{[0,1)^d} X_t \, \D t | \le
\int_{[0,1)^d} |X_t| \, \D t $; we get $ \log \Ex \exp \la
\int_{[0,1)^d} X_t \, \D t \le \la^2 $ for all $ \la \in [-1,1]
$. Thus, $ f(\la;1,\dots,1) \le \la^2 $ for these $\la$.
\end{proof}

\begin{proof}[Proof of Prop.~\ref{4.2}] \let\qed\relax
First, a general
fact. For arbitrary random variable $W$,
\begin{equation}
\text{if} \;\; \Ex \E^W \le \frac54 \;\; \text{and} \;\; \Ex \E^{-W} \le
\frac54 \, ,
\quad \text{then} \quad \Ex \E^{|W|} \le 2 \, ,
\end{equation}
since $ \E^{|W|} = \frac43 ( \E^{|W|} + \E^{-|W|} - 1) - \frac13 (
\E^{|W|/2} - 2 \E^{-|W|/2} )^2 \le \frac43 ( \E^{|W|} + \E^{-|W|} - 1) =
\frac43 ( \E^W + \E^{-W} - 1) $.

Therefore, by \eqref{***},
\begin{multline}\label{4.8}
\text{if} \;\; \Ex \E^W \le \frac54, \;\; \Ex \E^{-W} \le \frac54
 \;\; \text{and} \;\; \Ex W = 0 \, , \quad \text{then} \\
\forall \la \in [-1,1] \;\; \log \Ex \exp \la W \le \la^2  \, .
\end{multline}

Applying \eqref{4.8} to $ W = \frac{\eps}{\sqrt{r_1\dots r_d}}
\int_{[0,r_1)\times[0,r_d)} X_t \, \D t $ (and taking supremum in $X$)
we get, for arbitrary $ \eps>0 $,
\[
\text{if} \;\; \exp f(\pm\eps; r_1,\dots,r_d ) \le \frac54 \, , \quad
\text{then} \quad
\forall \la \in [-1,1] \;\; f(\eps\la; r_1,\dots,r_d) \le \la^2  \, ,
\]
that is,
\begin{equation}\label{4.9}
\forall \la \in [-\eps,\eps] \;\; \eps^2 f(\la; r_1,\dots,r_d) \le
\la^2  \, .
\end{equation}
Due to Lemma \ref{4.6}, for arbitrary $ r_1,\dots,r_d $ there exists $
\eps>0 $ such that $ f(\pm\eps; r_1,\dots,r_d ) \le \log\frac54 $, and
therefore \eqref{4.9} holds.

\begin{sloppypar}
In order to obtain a single $\eps$ for all $ r_1,\dots,r_d \in
[\frac12 C, C] $ we recall that $ f( \la \sqrt{r_1\dots r_d};
r_1,\dots,r_d ) $ is an increasing function of $ r_1,\dots,r_d $ (as
noted in the proof of Lemma \ref{4.3}). We take $\eps$ such that\\
$ \forall \la \in [-\eps,\eps] \;\; \eps^2 f(\la; C,\dots,C) \le
\la^2 $, that is, $ \eps^2 f(\la \sqrt{C^d}; C,\dots,C ) \le C^d \la^2
$ for $ |\la| \le C^{-d/2} \eps $; then $ \eps^2 f( \la \sqrt{r_1\dots
r_d}; r_1,\dots,r_d ) \le C^d \la^2 $ for $ |\la| \le C^{-d/2} \eps
$, that is, $ \eps^2 f(\la;r_1,\dots,r_d) \le \frac{C^d}{r_1\dots r_d}
\la^2 $ for $ |\la| \le \sqrt{r_1\dots r_d} C^{-d/2} \eps $, which
implies $ (2^{-d/2} \eps )^2 f(\la;r_1,\dots,r_d) \le \la^2 $ for $
|\la| \le 2^{-d/2} \eps $. \qquad\qquad\qquad\qquad\qquad\qquad\rlap{$\qedsymbol$}
\end{sloppypar}
\end{proof}

We summarize.

\begin{theorem}\label{4.10}
The function $f$ defined by \eqref{4.1} satisfies condition $\A_d$ for
all $d$.
\end{theorem}

\begin{corollary}\label{4.11}
For each $d$ there exists $C_d$ such that $ f(\la;r_1,\dots,r_d) \le
C_d \la^2 $ whenever $ C_d |\la| \le \sqrt{r_1\dots r_d} \log^{-d}
(r_1\dots r_d) $ and $ \min(r_1,\dots,r_d) \ge C_d $; here $f$ is
defined by \eqref{4.1}.
\end{corollary}

\begin{proof}
Combine theorems \ref{II-1.5} and \ref{4.10}.
\end{proof}

\section[Proving the main result]
  {\raggedright Proving the main result}
\label{sect5}
The term ``splittable'' is defined both here and in \cite{II}, and
the definitions are nonequivalent. In order to avoid ambiguity,
below we use either ``\splittable{C}'' (defined here by
Def.~\ref{1.4}(b) and undefined in \cite{II}) or ``uniformly
splittable'' (defined in \cite[Def.~1.4]{II} and undefined here).

Every \splittable{C} random field $X$ on $ \R^d $ leads to
a centered random field $ \hat X $ on
$ \R^d $ as defined in \cite[Sect.~1]{II}; it is just the family of
integrals $ \int_{[\al_1,\be_1)\times\dots\times[\al_d,\be_d)}
X_t \, \D t $ indexed by boxes $
[\al_1,\be_1]\times\dots\times[\al_d,\be_d] \subset \R^d $; that is,
\[
\int_{[\al_1,\be_1]\times\dots\times[\al_d,\be_d]} \hat X_t \, \D t
= \int_{[\al_1,\be_1)\times\dots\times[\al_d,\be_d)} X_t \, \D t
\]
whenever $ \al_1<\be_1, \dots, \al_d<\be_d $. (If puzzled, note that
the left-hand side is nothing but a conventional notation used
in \cite{II} and \cite{III}; half-open intervals could be used there equally well.)
The additivity \cite[(1.1)]{II} is evidently satisfied. The zero mean
condition \cite[(1.2)]{II} is ensured by Def.~\ref{1.3}(b).

If, in addition, $X$ is stationary, then $ \hat X $ is a centered
measurable stationary (``CMS'') random field on $ \R^d $ as defined
in \cite[Sect.~1]{II} (see also \cite[Sect.~1]{III}).
Stationarity is evident.

\begin{lemma}
The measurability condition \cite[(1.3)]{II} is satisfied.
\end{lemma}

\begin{proof}
For arbitrary $ u_1,\dots,u_d, \, v_1,\dots, v_d \in [0,\infty) $ the
symmetric difference $ B \De B_0 $ between boxes $ B =
[\al_1-u_1,\be_1-v_1)\times\dots\times[\al_d-u_d,\be_d-v_d) $ and $
B_0 = [\al_1,\be_1)\times\dots\times[\al_d,\be_d) $ satisfies $ B \De
B_0 = ( B \cup B_0 ) \setminus ( B \cap B_0 ) \subset B_2 \setminus
B_1 $ where $ B_2 =
[\al_1-u_1,\be_1)\times\dots\times[\al_d-u_d,\be_d) $ and $ B_1 =
[\al_1,\be_1-v_1)\times\dots\times[\al_d,\be_d-v_d) $. Also, $
B_2 \setminus B_1 \subset (B_2\setminus B) \cup (B\setminus B_1) $
(equal, in fact). Thus, $ | \int_B X_t \, \D t - \int_{B_0} X_t \, \D
t | \le \int_{B_2\setminus B} |X_t| \, \D t +  \int_{B\setminus B_1}
|X_t| \, \D t $. It follows that $ \int_B X_t \, \D t $ converges to
$ \int_{B_0} X_t \, \D t $ (as $ u_1,\dots,u_d, v_1,\dots,v_d \to 0+
$) almost surely, therefore, in distribution. By stationarity,
$ \int_{[\al_1,\be_1+u_1-v_1) \times \dots \times
[\al_d,\be_d+u_d-v_d)} X_t \, \D t $ is distributed like $ \int_B
X_t \, \D t $ and therefore converges to $ \int_{B_0}
X_t \, \D t $ in distribution. It shows that
$ \int_{[\al_1,\be_1+w_1)\times\dots\times[\al_d,\be_d+w_d)} X_t \, \D
t $ converges to $ \int_{B_0} X_t \, \D t $ in distribution as $
w_1,\dots,w_d \to 0 $ (not ``$\to0+$'' this time). Thus, the
distribution of $ \int_{[0,r_1)\times\dots\times[0,r_d)} X_t \, \D t $
as a function of $ r_1,\dots,r_d $ is continuous, therefore Borel
measurable.
\end{proof}

For every $ d \in \{1,2,\dots\} $ and $ C \in (0,\infty) $ we consider
the set $ \Spl_C (d) $ of all \splittable{C} random fields
on $ \R^d $, and the family $ ( \hat X )_{X\in\Spl_C(d)} $ of the
corresponding centered random fields on $ \R^d $.

Uniform splittability is defined (in \cite{II}) by recursion in
dimension $ d \in \{0,1,2,\dots\} $, treating a centered random field on
$ \R^0 $ as just a single random variable of zero mean. Accordingly we
define $ \Spl_C(0) $ as the set of all random variables $X$ such that
$ \Ex \exp \frac1C |X| \le 2 $ and $ \Ex X = 0 $. (See also the last
line of Def.~\ref{1.3}(c).) For $ X \in \Spl_C(0) $ we define just
$ \hat X = X $.

\begin{proposition}\label{5.2}
For every $ d \in \{0,1,2,\dots\} $ and $ C \in (0,\infty) $ the family
$ ( \hat X )_{X\in\Spl_C(d)} $ is uniformly splittable.
\end{proposition}

\begin{lemma}\label{5.3}
Let $X$ be a random field on $ \R^d $, $ (X^-,X^+,X^0) $ a split of
$X$, and $Y$ the leak of this split. Then, in terms of \cite{II}, $
(\widehat{X^-},\widehat{X^+},\widehat{X^0}) $ is a split of $ \hat X
$, and its leak along the hyperplane $ \{0\}\times\R^{d-1} $ on a
strip $ [a,b] \times \R^{d-1} $ (for arbitrary $ a<0 $, $ b>0 $) is
$ \hat U $ where $ U $ is a random field on $ \R^{d-1} $ defined by $
U_{t_2,\dots,t_d} = \int_{[a,b)} Y_{t_1,\dots,t_d} \, \D t_1 $.
\end{lemma}

\begin{proof}\footnote{%
 The proof is straightforward and boring, but written out anyway.}
Clearly, $ (\widehat{X^-},\widehat{X^+},\widehat{X^0}) $ is a split of
$ \hat X $. Its leak $Z$ on the given strip is defined
(in \cite[p.~2]{II}) by
\begin{multline*}
\int_{[\al_2,\be_2]\times\dots\times[\al_d,\be_d]} Z(t) \, \D t =
 \int_{[a,0]\times[\al_2,\be_2]\times\dots\times[\al_d,\be_d]}
  \widehat{X^-}(t) \, \D t \; + \\
+ \int_{[0,b]\times[\al_2,\be_2]\times\dots\times[\al_d,\be_d]}
  \widehat{X^+}(t) \, \D t -
 \int_{[a,b]\times[\al_2,\be_2]\times\dots\times[\al_d,\be_d]}
  \widehat{X^0}(t) \, \D t \, .
\end{multline*}
The right-hand side is
\begin{multline*}
 \int_{[a,0)\times[\al_2,\be_2)\times\dots\times[\al_d,\be_d)}
   X^-(t) \, \D t \; +
  \int_{[0,b)\times[\al_2,\be_2)\times\dots\times[\al_d,\be_d)}
   X^+(t) \, \D t - \\
- \int_{[a,b)\times[\al_2,\be_2)\times\dots\times[\al_d,\be_d)}
   X^0(t) \, \D t =
 \int_{[a,b)\times[\al_2,\be_2)\times\dots\times[\al_d,\be_d)}
  Y(t) \, \D t
\end{multline*}
(by Def.~\ref{1.2}). On the other hand,
\begin{multline*}
\int_{[\al_2,\be_2]\times\dots\times[\al_d,\be_d]} \hat U(s) \, \D s =
 \int_{[\al_2,\be_2)\times\dots\times[\al_d,\be_d)} U(s) \, \D s = \\
= \int_{[a,b)\times[\al_2,\be_2)\times\dots\times[\al_d,\be_d)} Y_t \, \D t
\end{multline*}
whenever $ \al_2<\be_2, \dots, \al_d<\be_d $. Thus, $ Z = \hat U $.
\end{proof}

\begin{proof}[Proof of Prop.~\ref{5.2}]
Induction in $ d $. For $ d=0 $ we take $ \eps=1/C $ and get
$ \Ex \exp \eps |\hat X| \le 2 $ as required. For $ d \ge 1
$, Condition (a) of \cite[Def.~1.4]{II} (there take $ B = [0,1]^d $
and $ \eps=1/C $) follows from Condition (a) of Def.~\ref{1.3} here
(for $ \frac1C X $ and $ k=d $).

In order to check Condition (b) of \cite[Def.~1.4]{II} (for $ d \ge 1
$) we need splits of $ \hat X $ (indexed by $ k \in \{1,\dots,d\} $
and $ r \in \R $)
whose leaks (as defined there) are members of the family $ ( \hat Y
)_{Y\in\Spl_C(d-1)} $; this is sufficient, since this family is
uniformly splittable by the induction hypothesis.

WLOG, $ k=1 $ (the first coordinate) since,
first, $ \Spl_C(d) $ and $ \Spl_C(d-1) $ are permutation invariant (by
Corollary \ref{2.6}), and second, in \cite{II}, the class of all centered
random fields is permutation invariant, and leaks on different
coordinates turn into one another under coordinate permutations.
Likewise, WLOG, $ s=0 $ (use shift invariance).

Lemma \ref{2.12b} (in combination with Remark \ref{2.12c}) gives a
leak $ Y $ for $ X $ such that for all $ a<0 $, $ b>0 $ the random field $ U
$ on $ \R^{d-1} $ defined by $ U_{t_1,\dots,t_{d-1}} = \int_{[a,b)}
Y_{s,t_1,\dots,t_{d-1}} \, \D s $ is \splittable{C}, that is, belongs
to $ \Spl_C(d-1) $. By Lemma \ref{5.3}, the
leak of the corresponding split of $ \hat X $ along the hyperplane
$ \{0\}\times\R^{d-1} $ on $ [a,b]\times\R^{d-1} $ is $ \hat U $.
\end{proof}

\begin{corollary}\label{5.4}
(a) For every splittable\footnote{%
 According to Def.~\ref{1.4}(c).}
random field $X$ on $\R^d$, the corresponding
centered random field $ \hat X $ is splittable (in the sense
of \cite{II}, \cite{III}).

(b) For every splittable\footnote{%
 According to Def.~\ref{1.4}(c).}
stationary random field $X$ on $\R^d$, the
corresponding CMS random field $ \hat X $ is splittable (in the sense
of \cite{II}, \cite{III}).
\end{corollary}

Our main results (Theorem \ref{theorem1} and Corollaries \ref{1.6}
and \ref{1.7}) are formulated for compactly supported continuous
functions $ \phi : \R^d \to \R $. More generally, we may try compactly
supported bounded Borel measurable functions $ \phi : \R^d \to \R $.
We introduce the space $ \BB(d) $ of all such functions, and
its subset $ \G(d) $ of all functions $ \phi \in \BB(d) $ such that the
claim of Theorem \ref{theorem1} holds for $\phi$. If
$ \phi\in\G(d) $, then the shifted function $ t \mapsto \phi(t+s) $
belongs to $ \G(d) $ (for every $ s \in \R^d $) due to stationarity.
If $ \phi \in \G(d) $ and $ c \in \R $, then $ c\phi \in \G(d) $,
since $ cX $ is splittable whenever $X$ is. Also, invariance of
$\G(d)$ under permutations of coordinates follows from such invariance
of splittability.

By a box (in $\R^d$) we mean here a set of the form $
[\al_1,\be_1) \times \dots \times [\al_d,\be_d) $ for
$ \al_1<\be_1, \dots, \al_d<\be_d $.

\begin{proposition}\label{5.5}
For every box $ B \subset \R^d $, its indicator function $ \One_B $
belongs to $ \G(d) $.
\end{proposition}

\begin{proof} \let\qed\relax
Corollary \ref{5.4}(b) in combination with \cite[Theorem 1.1]{III}
gives
\[
\frac1{r_1\dots
r_d \la^2} \log \Ex \exp \la \int_{[0,r_1]\times\dots\times[0,r_d]} \hat
X_t \, \D t \to \frac{\si^2}2
\]
as $ r_1,\dots,r_d \to \infty $, $ \la \to 0 $, $ \la \log^d(r_1\dots
r_d) \to 0 $. In particular, for arbitrary $ r_1,\dots,r_d > 0 $,
\[
\frac1{(r_1 r)\dots(r_d r) \la^2} \log \Ex \exp \la \int_{[0,r_1
r)\times\dots\times[0,r_d r)} X_t \, \D t \to \frac{\si^2}2
\]
as $ r \to \infty $, $ \la \to 0 $, $ \la \log^d r \to 0 $. Thus
(using stationarity),
\[
\frac1{r^d\cdot\vol B\cdot\la^2} \log \Ex \exp \la \int_{rB}
X_t \, \D t \to \frac{\si^2}2
\]
as $ r \to \infty $, $ \la \to 0 $, $ \la \log^d r \to 0 $. That is,
for $ \phi = \One_B $ and $ \si_X = \si $,
\[
\qquad \frac1{r^d \la^2} \log \Ex \exp \la \int_{\R^d} \phi \Big( \frac1r
t \Big) X_t \, \D t \to \frac{\si_X^2}2 \int_{\R^d} \phi^2(t) \, \D
t \, . \qquad\qquad\qquad\rlap{$\qedsymbol$}
\]
\end{proof}

\begin{proposition}\label{5.6}
Let $ \phi \in \BB(d) $, and the two functions
$ \phi \cdot \One_{(-\infty,0)\times\R^{d-1}} $,
$ \phi \cdot \One_{[0,\infty)\times\R^{d-1}} $ belong to $
\G(d) $. Then $ \phi \in \G(d) $.\,\footnote{%
 See also \cite[Lemma 2.14]{2008}.}
\end{proposition}

Given $ \phi \in \BB(d) $ and a stationary splittable random field $X$
on $\R^d$, we introduce for convenience $ I(r)
= \int_{\R^d} \phi \( \frac1r t \) X_t \, \D t $ and similarly $
I_-(r) = \int_{(-\infty,0)\times\R^{d-1}} \phi \( \frac1r t \)
X_t \, \D t $, $ I_+(r)
= \int_{[0,\infty)\times\R^{d-1}} \phi \( \frac1r t \) X_t \, \D t $.

\begin{lemma}\label{5.7}
\begin{multline*}
\limsup_\myatop{ r\to\infty, \la\to0 }{ \la\log^d r \to 0 } \frac1{ r^d \la^2}
 \log \Ex \exp \la I(r) \le \\
\le \limsup_\myatop{ r\to\infty, \la\to0 }{ \la\log^d r \to 0 } \frac1{ r^d \la^2}
  \log \Ex \exp \la I_-(r)
 + \limsup_\myatop{ r\to\infty, \la\to0 }{ \la\log^d r \to 0 } \frac1{ r^d \la^2}
  \log \Ex \exp \la I_+(r) \, .
\end{multline*}
\end{lemma}

\begin{proof}
We take $ C\in(0,\infty) $ such that $ \phi X $ is \splittable{C}
by Corollary \ref{2.3}, and $ a\in(0,\infty) $ such that
$ \phi(\cdot)=0 $ outside $ [-a,a)^d $. We apply Lemma \ref{4.1a} (the
upper bound) to
the \splittable{1} random field $ \( \frac1C \phi\(\frac1r t\)
X_t\)_{t\in\R^d} $ and the box $ B = [-ra,ra)^{d-1} $ (not of the form
$ [0,r_2)\times\dots[0,r_d) $, which is harmless due to shift
invariance of splittability); we get
\begin{multline*}
\log \Ex \exp \frac\la{C} I(r)
\le \frac1p \log \Ex \exp \frac{p\la}{C} I_-(r) +
 \frac1p \log \Ex \exp \frac{p\la}{C} I_+(r) + \\
+ \frac{p-1}p f \Big(\! -\frac{p}{p-1} \la
 (2ra)^{(d-1)/2}; \underbrace{2ra,\dots,2ra}_{d-1} \Big) \, .
\end{multline*}
By Corollary \ref{4.11}, the last term (in the right-hand side) does
not exceed $ \frac{p-1}p C_{d-1} \( \frac{p}{p-1}\la(2ra)^{(d-1)/2}\)^2 $ provided
that $ C_{d-1} \big| \frac{p}{p-1}\la(2ra)^{(d-1)/2} \big| \le \break
(2ra)^{(d-1)/2} \log^{-d} (2ra)^{d-1} $ and $ 2ra \ge C_{d-1} $. When
$r$ is large enough and $ |\la| \log^d r $ is small enough, we get
$ \cO(r^{d-1}\la^2) $ for that last term, and therefore
\begin{multline*}
\frac1{r^d\la^2} \log \Ex \exp \frac\la{C} I(r) \le \\
\le \frac1p \cdot \frac1{r^d\la^2} \log \Ex \exp \frac{p\la}{C} I_-(r) +
 \frac1p \cdot \frac1{r^d\la^2} \log \Ex \exp \frac{p\la}{C} I_+(r)
 + \cO\Big(\frac1r\Big) \, ,
\end{multline*}
that is (substituting $ C\la $ for $\la$ and multiplying by $ C^2 $)
\begin{multline*}
\frac1{r^d\la^2} \log \Ex \exp \la I(r) \le \\
\le \frac1p \cdot \frac1{r^d\la^2} \log \Ex \exp p\la I_-(r) +
 \frac1p \cdot \frac1{r^d\la^2} \log \Ex \exp p\la I_+(r)
 + \cO\Big(\frac1r\Big) \, .
\end{multline*}
Taking into account that (substituting $ \frac\la{p} $ for $ \la $)
\[
\frac1p \limsup_\myatop{ r\to\infty, \la\to0 }{ \la\log^d r \to 0
} \frac1{ r^d \la^2} \log \Ex \exp p\la I_-(r) = p \limsup_\myatop{
r\to\infty, \la\to0 }{ \la\log^d r \to 0 } \frac1{
r^d \la^2} \log \Ex \exp \la I_-(r)
\]
and the same for $ I_+ $, we get
\begin{multline*}
\limsup_\myatop{ r\to\infty, \la\to0 }{ \la\log^d r \to 0 } \frac1{ r^d \la^2}
 \log \Ex \exp \la I(r) \le \\
\le p \limsup_\myatop{ r\to\infty, \la\to0 }{ \la\log^d r \to 0 } \frac1{ r^d \la^2}
 \log \Ex \exp \la I_-(r) +
p \limsup_\myatop{ r\to\infty, \la\to0 }{ \la\log^d r \to 0 } \frac1{ r^d \la^2}
 \log \Ex \exp \la I_+(r)
\end{multline*}
for all $ p>1 $, therefore, also for $ p=1 $.
\end{proof}

\begin{proof}[Proof of Prop.~\ref{5.6}]
Lemma \ref{5.7} gives the upper bound
\begin{multline*}
\limsup_\myatop{ r\to\infty, \la\to0 }{ \la\log^d r \to 0 } \frac1{ r^d \la^2}
 \log \Ex \exp \la I(r) \le \\
\le \frac{\si_X^2}2 \int_{(-\infty,0)\times\R^{d-1}} \phi^2(t) \, \D t +
\frac{\si_X^2}2 \int_{[0,\infty)\times\R^{d-1}} \phi^2(t) \, \D t =
\frac{\si_X^2}2 \int_{R^d} \phi^2(t) \, \D t \, .
\end{multline*}
The lower bound
\begin{multline*}
\liminf_\myatop{ r\to\infty, \la\to0 }{ \la\log^d r \to 0 } \frac1{ r^d \la^2}
 \log \Ex \exp \la I(r) \ge \\
\ge \liminf_\myatop{ r\to\infty, \la\to0 }{ \la\log^d r \to 0 } \frac1{ r^d \la^2}
 \log \Ex \exp \la I_-(r) +
 \liminf_\myatop{ r\to\infty, \la\to0 }{ \la\log^d r \to 0 } \frac1{ r^d \la^2}
  \log \Ex \exp \la I_+(r) = \\
= \frac{\si_X^2}2 \int_{(-\infty,0)\times\R^{d-1}} \phi^2(t) \, \D t +
\frac{\si_X^2}2 \int_{[0,\infty)\times\R^{d-1}} \phi^2(t) \, \D t =
\frac{\si_X^2}2 \int_{R^d} \phi^2(t) \, \D t
\end{multline*}
is obtained similarly (via the lower bound in Lemma \ref{4.1a}).
\end{proof}

A linear combination of (finitely many) indicators of boxes will be
called a step function.

\begin{lemma}\label{5.8}
Every step function belongs to $ \G(d) $.
\end{lemma}

\begin{proof}
A linear combination of indicators of two disjoint boxes belongs to
$\G(d)$ by Propositions \ref{5.5} and \ref{5.6}. (Prop.~\ref{5.6} is
formulated for the hyperplane $ \{0\}\times\R^{d-1} $, but holds for
every $ \R^{k-1}\times\{s\}\times\R^{d-k} $ due to shift and
permutation invariance of $\G(d)$.) Continuing this way, every step
function can be obtained in finitely many steps.\footnote{%
 No pun intended\dots}
\end{proof}

\begin{lemma}\label{5.9}
Let $ \phi, \phi_1, \phi_2, \dots \in \BB(D) $, $ \sup_{t\in\R^d}
| \phi_n(t) - \phi(t) | \to 0 $ as $ n \to \infty $, and $ \phi_n $
be uniformly compactly supported (that is, they all vanish outside a
single bounded set). If $ \phi_1, \phi_2, \dots \in \G(d) $, then
$ \phi \in \G(d) $.\,\footnote{%
 See also \cite[Lemma 2.17]{2008}.}
\end{lemma}

The proof is somewhat similar to that of Lemma \ref{5.7}.

\begin{proof}
Given a splittable $X$, for arbitrary $ r \in (0,\infty) $ we consider
$ I(r) = \int_{\R^d} \phi \( \frac1r t \) X_t \, \D t $ and
$ I_n(r) = \int_{\R^d} \phi_n \( \frac1r t \) X_t \, \D t $.
WLOG, $X$ is \splittable{1} (otherwise, for a \splittable{C} $X$, turn
to $ C\phi $, $ C\phi_n $ and $ \frac1C X $).
We take $ a\in(0,\infty) $ such that $ \forall n \; \phi_n(\cdot)=0 $
outside $ [-a,a)^d $. Given $ \eps>0 $, we take $n_\eps$ such that
$ \forall t \; |\phi_n(t)-\phi(t)| \le \eps $ whenever $
n \ge n_\eps $. Lemma \ref{4.1ab} applied to $ p = \frac1{1-\eps} $
and random variables $ \la I_n(r) $ and $ \la \(I(r)-I_n(r)\) $ for
arbitrary $ \la\in\R $ gives
\begin{multline*}
\frac1{1-\eps} \log \Ex \exp (1-\eps) \la I_n(r)
- \frac\eps{1-\eps} \log \Ex \exp \Big(\!
 -\frac{1-\eps}\eps \la \(I(r)-I_n(r)\) \Big) \le \\
\le \log \Ex \exp \la I(r) \le (1-\eps) \log \Ex \exp \frac\la{1-\eps}
 I_n(r) + \eps \log \Ex \exp \frac\la\eps \( I(r)-I_n(r) \) \, . 
\end{multline*}
By Corollary \ref{4.11} (and \ref{2.3}),
\begin{multline*}
\!\!\!\!\!\! \log \Ex \exp \frac\la\eps \( I(r)-I_n(r) \)
= \log \Ex \exp \la \int_{[-ra,ra)^d} \frac1\eps \bigg( \phi \Big( \frac1r
t \Big) - \phi_n \Big( \frac1r t \Big) \bigg) X_t \, \D t \le \\
\le f \( (2ra)^{d/2} \la; 2ra,\dots,2ra \) \le C_d (2ra)^d \la^2
\end{multline*}
whenever $ n \ge n_\eps $, $ C_d | (2ra)^{d/2} \la | \le
(2ra)^{d/2} \log^{-d} (2ra)^d $ and $ 2ra \ge C_d $. When $n$ and
$r$ are large enough and $ |\la| \log^d r $ is small enough, we have
\[
\frac1{r^d \la^2} \log \Ex \exp \frac\la\eps \( I(r)-I_n(r) \) \le C_d
(2a)^d \, ,
\]
and similarly,
\[
\frac1{r^d \la^2} \log \Ex \exp \Big(\! -\frac{1-\eps}\eps \la
\(I(r)-I_n(r)\) \Big) \le (1-\eps)^2 C_d (2a)^d \, .
\]
Taking into account that
\begin{multline*}
\lim_\myatop{ r\to\infty, \la\to0 }{ \la\log^d r \to 0 } \frac1{ r^d \la^2}
 \log \Ex \exp (1-\eps) \la I_n(r) = \\
\!\! = (1-\eps)^2 \!\! \lim_\myatop{ r\to\infty, \la\to0 }{ \la\log^d r \to
0 } \frac1{ r^d ((1-\eps)\la)^2}
 \log \Ex \exp (1-\eps) \la I_n(r) =
 (1-\eps)^2 \frac{\si_X^2}2 \int_{\R^d} \phi_n^2(t) \, \D t \, ,
\end{multline*}
and similarly
\[
\lim_\myatop{ r\to\infty, \la\to0 }{ \la\log^d r \to 0 } \frac1{ r^d \la^2}
\log \Ex \exp \frac\la{1-\eps} I_n(r) =
\frac1{(1-\eps)^2} \frac{\si_X^2}2 \int_{\R^d} \phi_n^2(t) \, \D t \, ,
\]
we have
\begin{multline*}
(1-\eps) \frac{\si_X^2}2 \int_{\R^d} \phi_n^2(t) \, \D t
 - \eps(1-\eps)C_d(2a)^d \le \\
\le \liminf_\myatop{ r\to\infty, \la\to0 }{ \la\log^d r \to 0 } \frac1{ r^d \la^2}
  \log \Ex \exp \la I(r)
 \le \limsup_\myatop{ r\to\infty, \la\to0 }{ \la\log^d r \to 0 } \frac1{ r^d \la^2}
  \log \Ex \exp \la I(r) \le \\
\le \frac1{1-\eps} \frac{\si_X^2}2 \int_{\R^d} \phi_n^2(t) \, \D t
 + \eps C_d(2a)^d
\end{multline*}
for all $ \eps>0 $ and $ n \ge n_\eps $. In the limit $ n\to\infty $,
$ \int_{\R^d} \phi_n^2(t) \, \D t $ turns into
$ \int_{\R^d} \phi^2(t) \, \D t $. Finally, take $ \eps\to0+ $.
\end{proof}

\begin{proof}[Proof of Theorem \ref{theorem1}]
Given a compactly supported continuous function $ \phi : \R^d \to \R
$, we take step functions $ \phi_n $ satisfying the conditions of
Lemma \ref{5.9}. By Lemma \ref{5.8}, each $ \phi_n $ belongs to
$ \G(d) $. By Lemma \ref{5.9}, $ \phi $ belongs to $ \G(d) $, which
means that the claim of Theorem \ref{theorem1} holds for $ \phi $.
\end{proof}

Corollaries \ref{1.6} and \ref{1.7} follow from Theorem \ref{theorem1}
in the same way as \cite[Corollaries 1.7, 1.8]{I} follow
from \cite[Theorem 1.6]{I}. (See also \cite[Sect.~6]{III}.)

\begin{remark}
More generally,
\[
\! \lim_\myatop{ r_1,\dots,r_d\to\infty, \la\to0 }{ \la\log^d
(r_1\dots r_d) \to 0 } \frac1{ r_1\dots r_d \la^2}
\log \Ex \!\exp \la \! \int_{\R^d} \phi \Big( \frac1{r_1}
t_1, \dots, \frac1{r_d} t_d \Big) X_t \, \D t
= \frac12 \|\phi\|^2 \si_X^2 \, ,
\]
where $ t = (t_1,\dots,t_d) \in \R^d $. The proof needs only trivial
modifications (starting from Prop.~\ref{5.5}).
\end{remark}

\bigskip
\filbreak
{
\small
\begin{sc}
\parindent=0pt\baselineskip=12pt
\parbox{4in}{
Boris Tsirelson\\
School of Mathematics\\
Tel Aviv University\\
Tel Aviv 69978, Israel
\smallskip
\par\quad\href{mailto:tsirel@post.tau.ac.il}{\tt
 mailto:tsirel@post.tau.ac.il}
\par\quad\href{http://www.tau.ac.il/~tsirel/}{\tt
 http://www.tau.ac.il/\textasciitilde tsirel/}
}

\end{sc}
}
\filbreak

\end{document}